\documentclass[12pt]{article}

\usepackage[margin=1in]{geometry}  
\usepackage{graphicx}              
\usepackage{amsmath}               
\usepackage{amsfonts}              
\usepackage{amsthm}                
\usepackage[below]{placeins} 
\usepackage{float}  
\usepackage{subfig} 

\theoremstyle{definition} 
\newtheorem{thm}{Theorem}[section]

\newtheorem{prop}[thm]{Proposition}
\newtheorem{defn}[thm]{Definition}

\newtheorem{exa}[thm]{Example}

\newtheorem{rem}[subsection]{Remark}

\bibliography{paper}

\begin{document}

\title{The Jones polynomial and Its Limitations\\(Master's Thesis)}

\author{Daniel Amankwah \\ 
African Institute for mathematical Sciences \\
Mbour, Senegal}

\maketitle

\begin{abstract}
  
  This paper will be an exposition of the Kauffman bracket polynomial model of the Jones polynomial, tangle methods for computing the Jones polynomial, and the use of these methods to produce non-trivial links that cannot be detected by the Jones polynomial.

\textbf{Keywords}: Links, Bracket polynomial, Jones polynomial. \\

This paper is the text of the author's
Master's Thesis written at the African Institute for mathematical Science under the supervision of Professor Louis H. Kauffman and Dr. Amadou Tall.
\end{abstract}

\section{Introduction}

In knot theory, the first polynomial invariant to be discovered was the Alexander polynomial by J.W. Alexander\cite{a6}. It was until $1985$ the New Zealand mathematician, Vaughan Jones announced a new invariant for knots and links called the Jones polynomial\cite{a5,a10}. He had noticed some relations in the field of operator theory that appeared similar to relations among knots. The Jones polynomial as an advantage over the Alexander is able to distinguish between a knot or link and it's mirror image   
 (i.e. detects chirality) and importantly it is able to detect more accurately whether a knot is knotted or a link is linked. A fundamental open problem in knot theory specifically the theory of Jones polynomial is as to whether there exists a non-trivial knot whose Jones polynomial is the same as the unknot. In this project we follow the method in \cite{a2} by Shalom Eliahou, Louis H. Kauffman and Morwen B. Thistlethwaite with more detail to construct links whose linking is not detectable by the Jones polynomial.

In section $2$, we discuss briefly some elementary definitions and notions in knot theory including some common knots invariants. We give an exposition to the Kauffman bracket model of the Jones polynomial in section $3$ and also discuss how the bracket applies to double stranded knots and links. In section $4$ we introduce some definitions and notations in tangle theory and apply tangle calculus in proving some relevant propositions associated with our method. We conclude by finally describing the method with an example to illustrate an infinite sequence of a 2-component link whose linking is not detectable by the Jones polynomial and also showing how the Thistlethwaite's link discovered by Morwen B. Thistlethwaite in \cite{a3b} fits into our construction.

\section{Elementary Terminologies}
\label{sect:basics}

\subsection{Knots and Links}

\begin{defn}[]
\label{defn:jwt}
A \textbf{knot} is an isotopy class of an embedding of a circle in $\mathbb{R}^{3}$. We assume that knots can be given a piecewise linear representation.
\end{defn}

\begin{defn}[]
\label{defn:jwt}
An \textbf{isotopy} from a knot or link $K$ to $K'$ is a family of embedding $K_{t}$, $0 \leq t \leq 1$ such that $K_{0}= K$ and $K_{1}=K'$.
\end{defn}

\begin{defn}[]
\label{defn:jwt}
Two knots $K$ and $K'$ are equivalent if there exists an orientation preserving homeomorphism, $h:\mathbb{R}^{3} \longrightarrow \mathbb{R}^{3}$ such that $h(K)=K'$. Here $K,K' \in \mathbb{R}^{3}$ are embedding of circle into the three-sphere $S^{3}$.
\end{defn}
It is a fact that two piecewise linear knots are isotopic if and only if they are equivalent. (See \cite{b1,b5})

\begin{defn}[]
\label{defn:jwt}
A \textbf{link} is the nonempty union of a finite number of disjoint knots.
\end{defn}

Since the components in a link are disjoint, the knots in a link are the connected components of that link. In particular, a knot is a special case of a link with one component. Some well-known links are shown in Figure ~\ref{eglink}.

\begin{figure}[!ht]
    \subfloat[hopf link\label{subfig-1:dummy}]{%
      \includegraphics[scale=0.09]{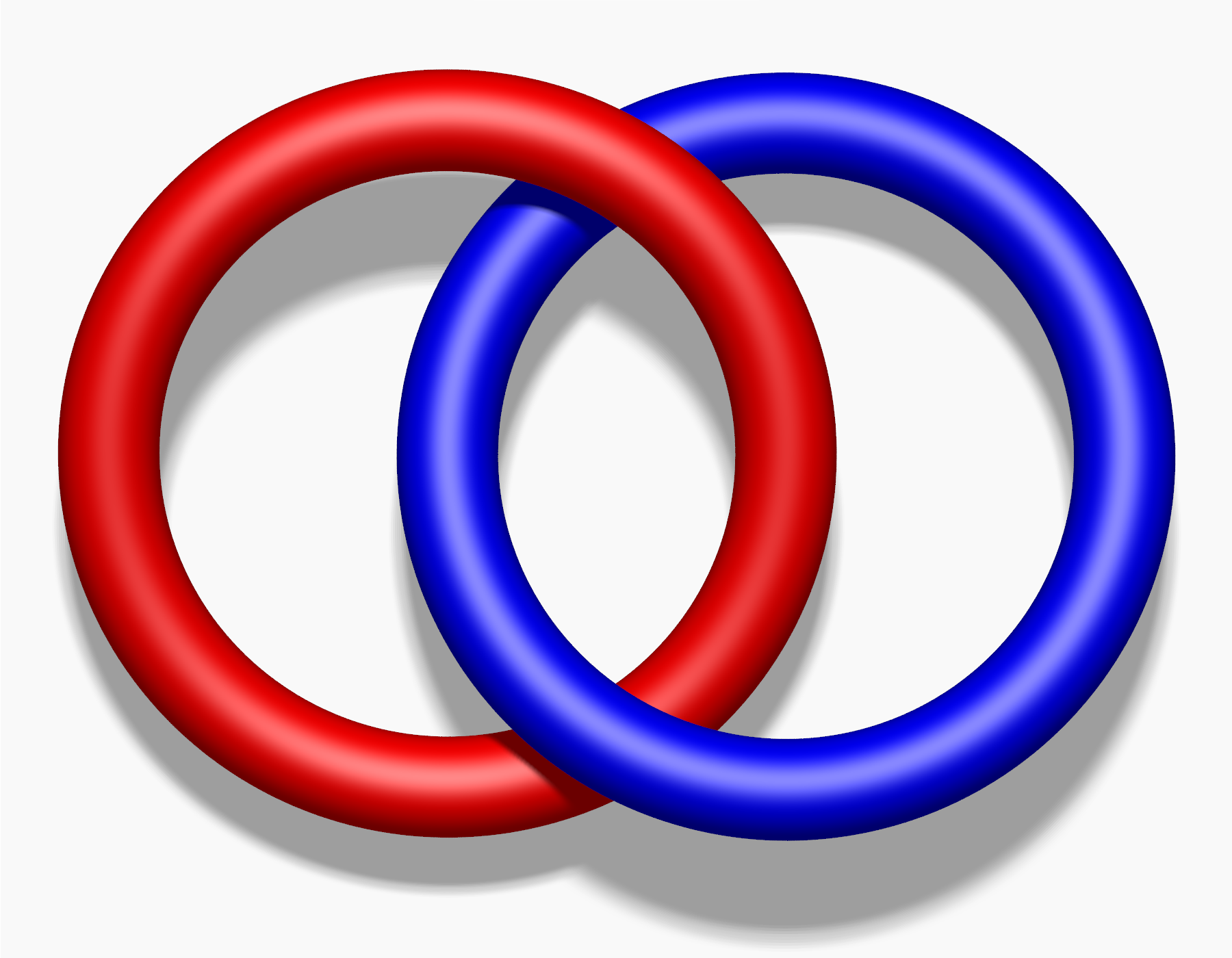}
    }
    \hfill
    \subfloat[Borromean rings\label{subfig-2:dummy}]{%
      \includegraphics[scale=0.08]{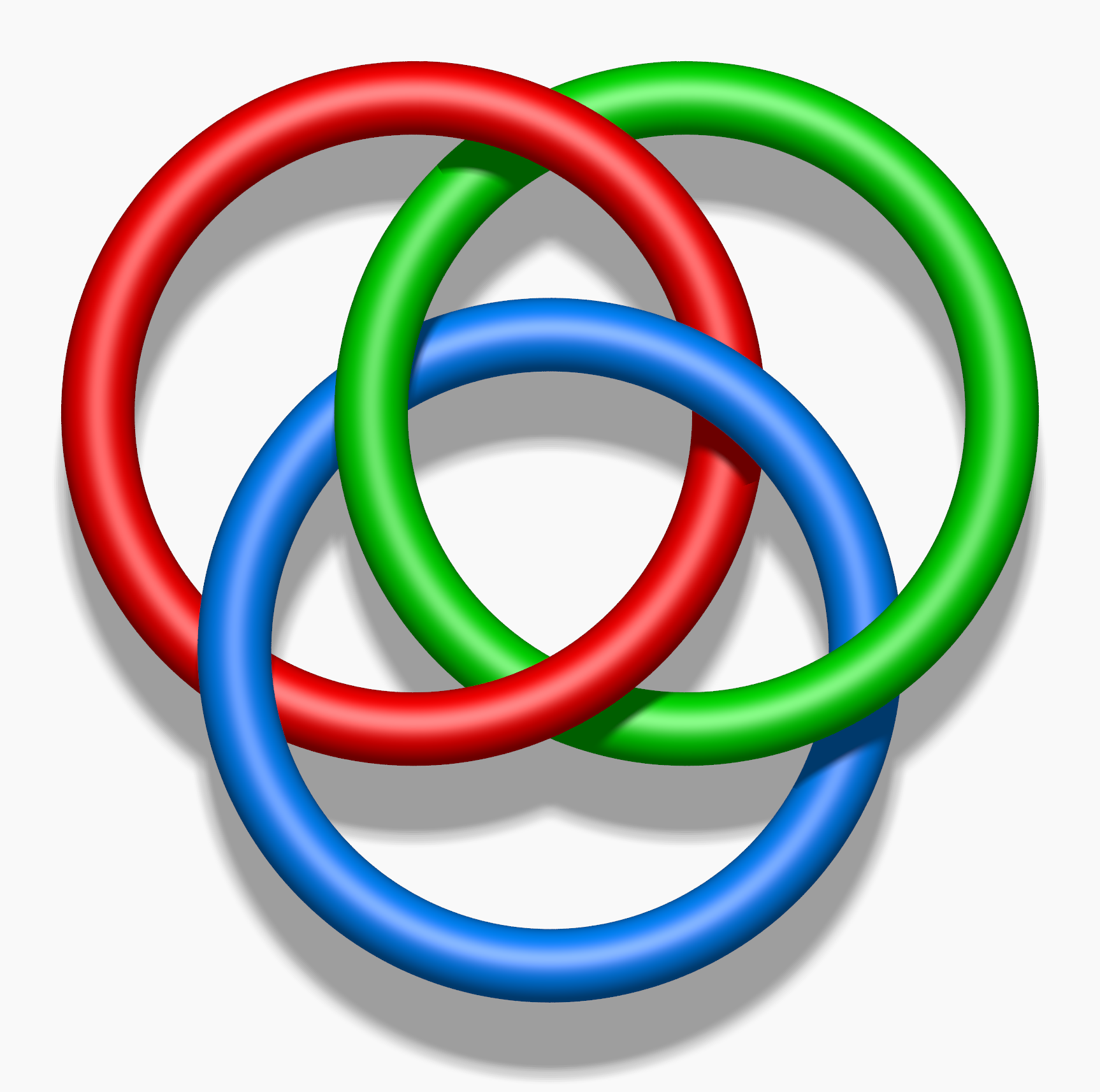}
    }
     \hfill
    \subfloat[Whitehead link\label{subfig-2:dummy}]{%
      \includegraphics[scale=0.6]{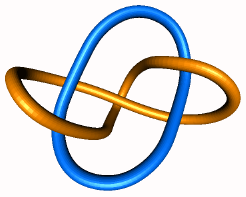}
    }
    \caption{Examples of link}
    \label{eglink}
\end{figure}

A useful way to visualize and manipulate a knot is to project it onto a plane.Think of the knot casting a shadow on that plane. A small change in the direction of projection will ensure that it is one-to-one except at the double points, called \textbf{crossings} where the "shadow" of the knot crosses itself once transversely. At each crossing, to be able to recreate the original knot, the over-strand must be distinguished from the under-strand. This is often done by creating a break in the strand going underneath. The resulting diagram is an immersed plane curve with the additional data of which strand is over and which strand is under at each crossing. These diagrams are called \textbf{knot diagrams} when they represent a knot and \textbf{link diagrams} when they represent a link.

\begin{defn}[]
\label{defn:jwt}
A \textbf{trivial knot} is a knot that is equivalent to a circle in a plane. A \textbf{trivial link} is a link that is equivalent to the union of disjoint circles lying in a plane.
\end{defn}

\subsection{Reidemeister moves}

\begin{defn}[]
\label{defn:jwt}
Consider a triangle $ABC$ with side $AC$ matching one of the line segments of a knot.In the plane determined by the triangle, we require that the region bounded $ABC$ intersects $K$ only in the edge $AC$. A \textbf{triangular detour} involves replacing the edge $AC$ of the knot $K$ with the two edges $AB$ and $BC$ to produce a new knot $L$. With the same notation, a \textbf{triangle shortcut} involves replacing the two edges $AB$ and $BC$ of knot $L$ with the single edge $AC$ to produce knot $K$. A \textbf{triangular move} is either a triangular detour or a triangular shortcut.
\end{defn}

\begin{figure}[!ht]
    \subfloat[$K$\label{subfig-1:dummy}]{%
      \includegraphics[scale=0.4]{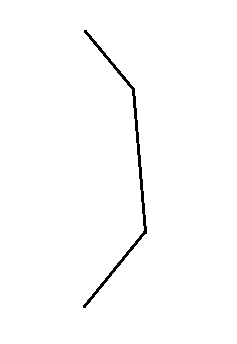}
    }
    \hfill
    \subfloat[A triangle on a line segment of the knot $K$\label{subfig-2:dummy}]{%
      \includegraphics[scale=0.4]{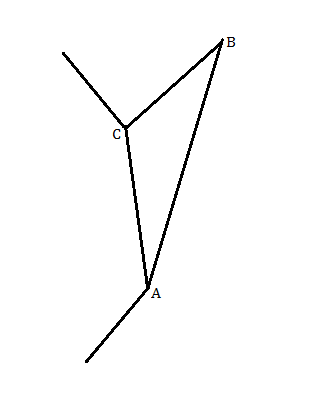}
    }
     \hfill
    \subfloat[$L$\label{subfig-2:dummy}]{%
      \includegraphics[scale=0.4]{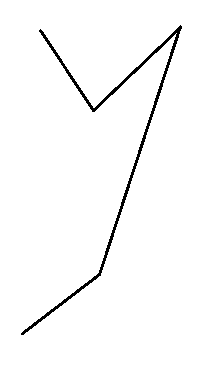}
    }
    \caption{A triangular detour on $K$; a triangular shortcut on knot $L$}
    \label{trimoves}
\end{figure}

Although a triangular move is a simple enough step in transforming one knot into an equivalent knot, such a change in a knot can significantly alter the diagram of the knot. We would like to consider triangular moves that only affect the knot diagram one vertex or one edge at a time. This enables us to analyze the effect of the move on some quantity we are trying to establish as a knot invariant.

In the 1920s Kurt Reidemester realized that the triangular moves can be broken down so the effect on the knot diagram is one of three simple modifications or their inverses. We designate the Reidemeister moves as Type $I$(add or remove a curl), Type $II$(remove or add two consecutive under cover crossings) and Type $III$(triangular move).

\begin{center}
\begin{enumerate}
\item[$I.$] $\raisebox{-1.0 em}{\includegraphics[scale=0.4]{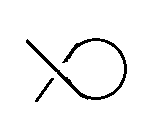}} \longleftrightarrow \raisebox{-0.8 em}{\includegraphics[scale=0.4]{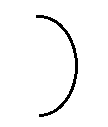}} \longleftrightarrow \raisebox{-0.8 em}{\includegraphics[scale=0.4]{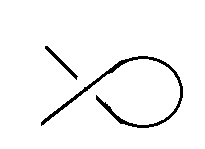}}$
\item[$II.$] $\raisebox{-0.8 em}{\includegraphics[scale=0.4]{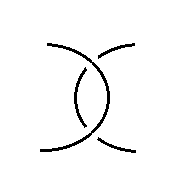}} \longleftrightarrow \raisebox{-0.8 em}{\includegraphics[scale=0.4]{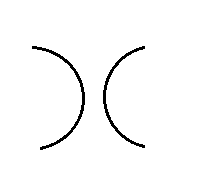}} \longleftrightarrow \raisebox{-1.2 em}{\includegraphics[scale=0.4]{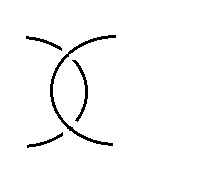}}$
\item[$III.$] $\raisebox{-0.8 em}{\includegraphics[scale=0.4]{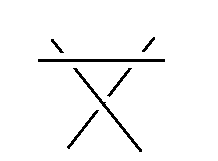}} \longleftrightarrow \raisebox{-0.8 em}{\includegraphics[scale=0.4]{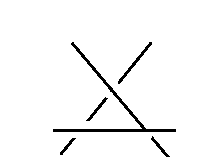}}$
\end{enumerate}
\end{center}

Reidemeister proved in the 1920s that these three moves(in conjunction with planar topological equivalence of the underlying universe) are sufficient to generate spatial isotopy. In other words Reidemeister proved that two knots(links) in space can be deformed into each other(ambient isotopy) if and only if their diagrams can be transformed into one another by planar isotopy and the three moves. By planar isotopy we mean a motion of the diagram in the plane that preserves the structure of the underlying universe.

\begin{defn}[]
\label{defn:jwt}
A \textbf{knot invariant} is a mathematical property or quantity associated with a knot that does not change as we perform Reidemeister moves on the knot.
\end{defn}

\begin{defn}[]
\label{defn:jwt}
An orientation of a link is a choice of direction to travel around each component of the link.Consider a crossing in a regular projection of an oriented link. Stand on the overpass and face in the direction of the orientation. The crossing is \textbf{right-handed} if and only if traffic on the underpass goes from right to left; the crossing is \textbf{left-handed} if and only if traffic on the underpass goes from left to right.In a regular projection of an oriented link of two components, we assign \textbf{$+1$} to right-handed crossing and \textbf{$-1$} to left-handed crossing.
\end{defn}

\begin{figure}[!h]
\centering 
\includegraphics[scale=0.8]{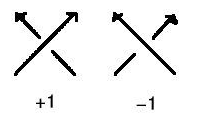}
\caption{Crossing signs}
\label{bandwidth} 
\end{figure}

Given an oriented link of $n$ components, one half of the sum of the signs of crossing with the exemption of self crossings is termed the \textbf{total linking number}.

Let $\alpha$ and $\beta$ be two components of a link. Denoting the set of crossing of the component $\alpha$ with the component $\beta$ by $\alpha \sqcap \beta$(thus $\alpha \sqcap \beta$ does not include self crossing of $\alpha$ or of $\beta$). Then the linking number of $\alpha$ and $\beta$ is defined by the formula: 

\begin{equation}
lk(\alpha,\beta)=\frac{1}{2}\sum_{p \in \alpha \sqcap \beta}\varepsilon(p).
\end{equation}
where  \[ \varepsilon(p)=  \left\{
\begin{array}{ll}
      +1 & p = \raisebox{-0.7 em}{\includegraphics[scale=0.3]{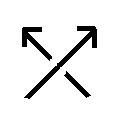}} \\
      -1 & p = \raisebox{-0.7 em}{\includegraphics[scale=0.3]{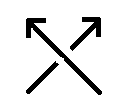}} \\
\end{array} 
\right. \]

\begin{exa}[\cite{a3}]
\label{exa:jwt}
The linking number of the Whitehead link(Figure ~\ref{whitehead}) below can be calculated as;
\end{exa}

\begin{figure}[!h]
\centering 
\includegraphics[scale=0.8]{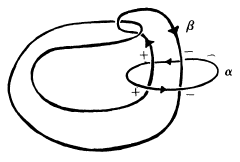}
\caption{}
\label{whitehead} 
\end{figure}

$$lk(\alpha,\beta)=\frac{1}{2}(1+1-1-1)=0$$

The above example gives the linking number of the Whitehead link, and even though it has linking number zero, it is linked.

\begin{rem}[]
\label{rem:jwt}
The linking number is an invariant of an oriented link of two components. We can check this by how Reidemeister moves affect the linking number.

A move of type $I$ generates a new crossing of a link component with itself. These are ignored in computing the linking number. Thus there is no change to this quantity.

A move of type $II$ generates two crossings, both involving the same component or components of the link. One crossing will be assigned $+1$ and the other will be assigned $-1$.Thus even if the crossing involve both components of the link, the sum of the numbers will not change.

A move of type $III$ merely alters the relative positions of the three crossing without changing the components or orientation involved. Hence, there is no change to the linking number
\end{rem}

\subsection{Alexander Polynomial}
In 1928, J.W Alexander introduced a polynomial invariant that is straight forward to compute,has a moderately easy proof and yet is fairly good at distinguishing among different knots. In his approach, the knot diagram gives entries of a matrix whose determinant is the \textbf{Alexander Polynomial}\cite{a6}. Alexander determined the entries of a matrix from the crossings and arcs of the diagram.

At each crossing of the knot diagram, the breaks in the lower strand divide the curve into arcs that begin and end at crossing. Since each arc has two ends and each crossing involves the ends of the arcs, the number of arcs is equal to the number of crossing. In our notation, we index the crossing $x_{1},x_{2},...,x_{n}$ and index the arcs $a_{1},a_{2},...,a_{n}$. Next we choose an orientation of the knot. As you travel around the knot, stand on the overpass of each crossing and face in the direction of the orientation. Label the three strands of the knot diagram that form this crossing: put $1-t$ on the overpass itself, label the end of the underpass arc on your left $t$ and label the end of the underpass arc on your right with -1. We then form an $n \times n$ matrix by writing the label on arc $a_{j}$ at crossing $x_{i}$ as $ij$-entry of the matrix. If an arc has more than one of its labels at the crossing, put the sum of the labels at the entry of the matrix. If the arc is not involved in forming the crossing, put a $0$ as the entry. This matrix is the crossing/arc matrix of the knot diagram. Next we delete one column and one row of the matrix and compute the determinant of the resulting $(n-1)\times(n-1)$ matrix. This will be a polynomial in the variable $t$ which depends on quite a few choices: the index system for the crossing and arcs, the orientation of the knot, the selection of a row and column to eliminate, to say nothing of the choice of the knot diagram. Alexander's amazing discovery is that although these choices may produce different polynomials, any one of them will be plus or minus a power of $t$ times any other. Thus if we normalize the polynomial to have a positive constant term, the resulting Alexander Polynomial will be a knot invariant.

In \cite{b2} one can find the proof showing that the Alexander polynomial does not depend on the indexing of crossings and arcs. It does not also depend on the row and column eliminated from the crossing/arc matrix. In the same text one can see how to use the Skein relation discovered by John Conway in $1969$ to compute the Alexander polynomial. The Alexander polynomial is invariant up to multiplication by $\pm t^{N}$ where $N$ is some integer.

\begin{exa}[\cite{b2}]
\label{exa:jwt}

\end{exa}

\begin{figure}[!h]
\centering 
\includegraphics[scale=0.8]{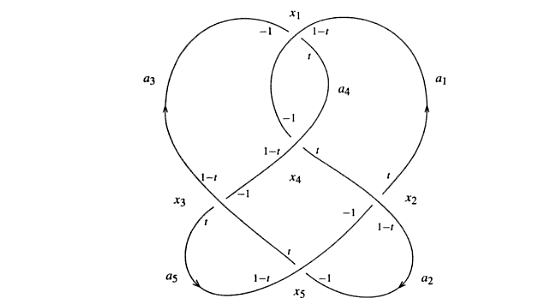}
\caption{}
\label{Alex} 
\end{figure}

The crossing/arc matrix of the above oriented knot diagram(Figure ~\ref{Alex} ) is given as:
 
 \begin{center}
$\begin{pmatrix}
1-t & 0 & -1 & t & 0 \\ 
t & 1-t & 0 & 0 & -1 \\ 
0 & 0 & 1-t & -1 & t \\ 
-1 & t & 0 & 1-t & 0 \\ 
0 & -1 & t & 0 & 1-t
\end{pmatrix}$ 
\end{center}

We then eliminate the fourth row and the last column and compute the determinant of the resulting $4 \times 4$ matrix as:

\begin{eqnarray*}
det\begin{pmatrix}
1-t & 0 & -1 & t \\ 
t & 1-t & 0 & 0 \\ 
0 & 0 & 1-t & -1 \\ 
0 & -1 & t & 0
\end{pmatrix}   & = &-t det \begin{pmatrix}
0 & -1 & t  \\ 
0 & 1-t & -1 \\ 
-1 & t & 0  
\end{pmatrix} +(1-t)det \begin{pmatrix}
1-t & 0 & 0  \\ 
0 & 1-t & -1 \\ 
-1 & t & 0  
\end{pmatrix}\\ & = &-t(t-1-t^{2})+t(1-t)^{2}  =  2t^{3}-3t^{2}+2t
\end{eqnarray*}

Finally we multiply the resulting determinant by $t^{-1}$ to normalize this polynomial in order to have a positive constant term. Thus \textbf{$2t^{2}-3t+2$} is the Alexander polynomial of this knot.

\section{The Kauffman Bracket Polynomial Model of the Jones Polynomial}
\label{sect:cross-ref}

In this chapter, we focus on the definition of the Jones polynomial as constructed by Louis Kauffman in terms of his bracket polynomial. The bracket polynomial is in general a Laurent polynomial since it contains terms with negative exponent.

\subsection{The Rules for the Kauffman Bracket Polynomial}

The Kauffman bracket polynomial of a regular projection of a link, say $L$ is denoted by $\langle L \rangle$ and $\langle L \rangle \in \mathbb{Z}[A,A^{-1}]$   \cite{a3}.
It satisfies the following rules;
\begin{enumerate}
\item The bracket of a diagram consisting of $n$ disjoint simple closed curves in the plane is $\delta^{n-1}$, where $\delta = -A^{-2}-A^{2}$.

\item $\left \langle \raisebox{-0.6 em}{\includegraphics[scale=0.3]{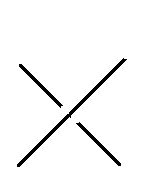}} \right \rangle = A\left \langle \raisebox{-0.6 em}{\includegraphics[scale=0.3]{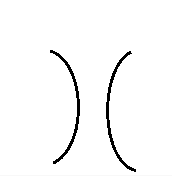}} \right \rangle + A^{-1}\left \langle \raisebox{-0.6 em}{\includegraphics[scale=0.3]{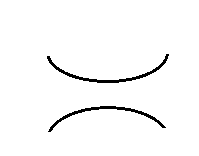}} \right \rangle$

$\left \langle \raisebox{-0.6 em}{\includegraphics[scale=0.3]{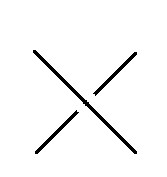}} \right \rangle = A\left \langle \raisebox{-0.6 em}{\includegraphics[scale=0.3]{t21.PNG}} \right \rangle + A^{-1}\left \langle \raisebox{-0.6 em}{\includegraphics[scale=0.3]{t24.PNG}} \right \rangle$

\item $\left \langle L \sqcup O \right \rangle = \delta \left \langle L \right \rangle $, where $O$ is a simple closed curve.
\end{enumerate}

A useful consequence of rule $2$ is the following switching formula;
\begin{equation}\label{switching}
A\left \langle \raisebox{-0.6 em}{\includegraphics[scale=0.3]{t56.PNG}} \right \rangle - A^{-1} \left \langle \raisebox{-0.6 em}{\includegraphics[scale=0.3]{t55.PNG}} \right \rangle =(A^{2}-A^{-2})\left \langle \raisebox{-0.6 em}{\includegraphics[scale=0.3]{t21.PNG}} \right \rangle
\end{equation}
This formula is useful since some computations will simplify quite quickly with the proper choices of switching and smoothing.

At this point we check the invariance of the bracket polynomial using the Reidemeister moves. We first consider the type II Reidemeister move;
\begin{eqnarray*}
\left \langle  \raisebox{-0.6 em}{\includegraphics[scale=0.4]{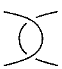}} \right \rangle   & = & A \left \langle  \raisebox{-0.6 em}{\includegraphics[scale=0.4]{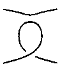}} \right \rangle + A^{-1} \left \langle  \raisebox{-0.6 em}{\includegraphics[scale=0.4]{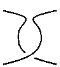}} \right \rangle \\
& = & A \left[A \left \langle  \raisebox{-0.6 em}{\includegraphics[scale=0.4]{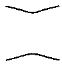}} \right \rangle + A^{-1} \left \langle  \raisebox{-0.6 em}{\includegraphics[scale=0.4]{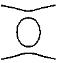}} \right \rangle \right] + A^{-1} \left[A \left \langle  \raisebox{-0.6 em}{\includegraphics[scale=0.4]{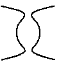}} \right \rangle + A^{-1} \left \langle  \raisebox{-0.6 em}{\includegraphics[scale=0.4]{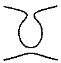}} \right \rangle \right]\\
& = & A^{2} \left \langle  \raisebox{-0.6 em}{\includegraphics[scale=0.4]{j11.PNG}} \right \rangle + \left \langle  \raisebox{-0.6 em}{\includegraphics[scale=0.4]{j8.PNG}} \right \rangle + \left \langle  \raisebox{-0.6 em}{\includegraphics[scale=0.4]{j9.PNG}} \right \rangle + A^{-2} \left \langle  \raisebox{-0.6 em}{\includegraphics[scale=0.4]{j10.PNG}} \right \rangle\\
& = & A^{2} \left \langle  \raisebox{-0.6 em}{\includegraphics[scale=0.4]{j11.PNG}} \right \rangle +(-A^{2}-A^{-2}) \left \langle  \raisebox{-0.6 em}{\includegraphics[scale=0.4]{j11.PNG}} \right \rangle + \left \langle  \raisebox{-0.6 em}{\includegraphics[scale=0.4]{j9.PNG}} \right \rangle + A^{-2} \left \langle  \raisebox{-0.6 em}{\includegraphics[scale=0.4]{j10.PNG}} \right \rangle\\
& = & (A^{2}-A^{2}-A^{-2}+A^{-2})\left \langle  \raisebox{-0.6 em}{\includegraphics[scale=0.4]{j11.PNG}} \right \rangle + \left \langle \raisebox{-0.6 em}{\includegraphics[scale=0.25]{t24.PNG}} \right \rangle\\
& = & \left \langle \raisebox{-0.6 em}{\includegraphics[scale=0.25]{t24.PNG}} \right \rangle
\end{eqnarray*}

Checking the invariance under type III moves is easy since we can use the invariance under type II moves to slide some arcs around a little.
\begin{eqnarray*}
\left \langle \raisebox{-0.6 em}{\includegraphics[scale=0.4]{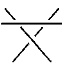}} \right \rangle & = & A \left \langle \raisebox{-0.6 em}{\includegraphics[scale=0.4]{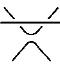}} \right \rangle + A^{-1} \left \langle \raisebox{-0.6 em}{\includegraphics[scale=0.4]{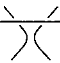}} \right \rangle \\
& = &  A \left \langle \raisebox{-0.6 em}{\includegraphics[scale=0.4]{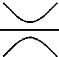}} \right \rangle + A^{-1} \left \langle \raisebox{-0.6 em}{\includegraphics[scale=0.4]{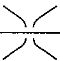}} \right \rangle \\
& = &  A \left \langle \raisebox{-0.6 em}{\includegraphics[scale=0.4]{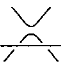}} \right \rangle + A^{-1} \left \langle \raisebox{-0.6 em}{\includegraphics[scale=0.4]{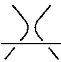}} \right \rangle \\
& = & \left \langle \raisebox{-0.6 em}{\includegraphics[scale=0.4]{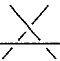}} \right \rangle \\
\end{eqnarray*}

\begin{defn}[]
\label{defn:jwt}
Two knot or link diagrams are defined to be \textbf{regular isotopic} if one can be obtained from the other by a combination of planar isotopy and the second and third Reidemeister moves.
\end{defn}

The bracket polynomial is therefore a regular isotopy invariant as it is not invariant under type I Reidemeister moves. This can be seen as follows;

For the curl with a right-handed crossing we have,
\begin{eqnarray*}
\left \langle \raisebox{-0.6 em}{\includegraphics[scale=0.4]{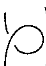}} \right \rangle & = & A \left \langle \raisebox{-0.6 em}{\includegraphics[scale=0.4]{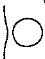}} \right \rangle + A^{-1} \left \langle \raisebox{-0.6 em}{\includegraphics[scale=0.4]{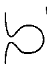}} \right \rangle \\
& = & A(-A^{-2}-A^{2}) \left \langle \raisebox{-0.6 em}{\includegraphics[scale=0.4]{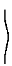}} \right \rangle + A^{-1} \left \langle \raisebox{-0.6 em}{\includegraphics[scale=0.4]{j3.PNG}} \right \rangle \\
& = & (-A^{3}-A^{-1}+A^{-1})\left \langle \raisebox{-0.6 em}{\includegraphics[scale=0.4]{j4.PNG}} \right \rangle \\
& = & -A^{3} \left \langle \raisebox{-0.6 em}{\includegraphics[scale=0.4]{j4.PNG}} \right \rangle 
\end{eqnarray*}

and for the curl with a left-handed crossing we have,
\begin{eqnarray*}
\left \langle \raisebox{-0.6 em}{\includegraphics[scale=0.4]{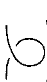}} \right \rangle & = & A \left \langle \raisebox{-0.6 em}{\includegraphics[scale=0.4]{j3.PNG}} \right \rangle + A^{-1} \left \langle \raisebox{-0.6 em}{\includegraphics[scale=0.4]{j2.PNG}} \right \rangle \\
& = & A \left \langle \raisebox{-0.6 em}{\includegraphics[scale=0.4]{j3.PNG}} \right \rangle + A^{-1}(-A^{-2}-A^{2})  \left \langle \raisebox{-0.6 em}{\includegraphics[scale=0.4]{j4.PNG}} \right \rangle \\
& = & (A-A-A^{-3})\left \langle \raisebox{-0.6 em}{\includegraphics[scale=0.4]{j4.PNG}} \right \rangle \\
& = & -A^{-3} \left \langle \raisebox{-0.6 em}{\includegraphics[scale=0.4]{j4.PNG}} \right \rangle 
\end{eqnarray*}

\begin{exa}[]
\label{exa:jwt}
The first diagram in Figure  ~\ref{trefoil}  indicated as $K$ is the right-handed trefoil. We wish to compute its bracket polynomial.
\end{exa}
\begin{figure}[!ht]
    \subfloat[$K$\label{subfig-1:dummy}]{%
    \includegraphics[scale=0.8]{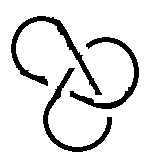}
    }
    \hfill
    \subfloat[$K'$\label{subfig-2:dummy}]{%
      \includegraphics[scale=0.8]{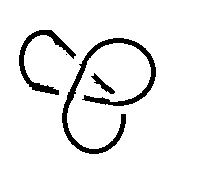}
    }
    \hfill
    \subfloat[$K''$\label{subfig-2:dummy}]{%
      \includegraphics[scale=0.8]{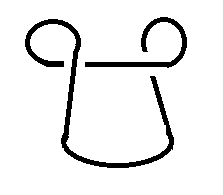}
    }
    \caption{}
    \label{trefoil}
  \end{figure}
  
By switching and smoothening one of its crossings, we can write a switching formula as in Equation ~\eqref{switching},
\begin{equation}
A^{-1}\langle K \rangle -A \langle K' \rangle=(A^{-2}-A^{2})\langle K'' \rangle
\end{equation}
Applying regular isotopy to the diagram $K'$ we obtain an unknot with a right-handed crossing, so then $\langle K' \rangle = -A^{3}$. Similarly the diagram $K''$ is an unknot with two left-handed crossings so then $\langle K'' \rangle = (-A^{-3})^{2}= A^{-6}$.We therefore obtain,
\begin{equation*}
A^{-1}\langle K \rangle -A(-A^{3})=(A^{-2}-A^{2})(A^{-6})
\end{equation*}
Hence
\begin{equation*}
A^{-1}\langle K \rangle =A^{-8}-A^{-4}-A^{4}
\end{equation*}
So
\begin{equation*}
\langle K \rangle = A^{-7}-A^{-3}-A^{5}
\end{equation*}

\begin{defn}[]
\label{defn:jwt}
The \textbf{writhe}, $w(L)$ of the regular projection $L$ of a link is the number of right-handed crossings minus the number of left-handed crossings.
\end{defn}

As we have seen, a type $I$ move that eliminates a right-handed crossing will introduce a factor $-A^{3}$ into the bracket and a type $I$ move that eliminates a left-handed crossing will introduce a factor  $-A^{-3}$ into the bracket.

The bracket polynomial of the regular projection $L$ of a link can thus be normalized by multiplying by the factor  $(-A)^{-3w(L)}$. The normalized bracket polynomial denoted by $X_{L}$ is therefore written as;
\begin{equation}
X_{L}=(-A)^{-3w(L)}\langle L \rangle
\end{equation}

Let $L$ and $L'$ differ by a single first Reidemeister move,
\begin{eqnarray*}
X_{L} & = & (-A)^{-3w(L)}\langle L \rangle \\
      & = & (-A)^{-3(w(L')+1)}(-A^{3})\langle L' \rangle\\
      & = & (-A)^{-3w(L')} \cdot (-A^{-3}) \cdot (-A^{3}) \langle L' \rangle \\
      & = & (-A)^{-3w(L')}\langle L' \rangle \\
      & = & X_{L'}
\end{eqnarray*}

We have shown that the bracket polynomial is invariant under type $II$ and type $III$ Reidemeister moves. Also by direct verification we can as well see that the writhe is invariant under type II and type III Reidemeister moves. We can therefore conclude that the normalized bracket polynomial is a knot invariant.

\begin{defn}[]
\label{defn:jwt}
The \textbf{Jones polynomial} $V_{L}$ of a link $L$ is obtained by making the substitution $A=t^{-\frac{1}{4}}$ in the polynomial $X_{L}$.
\end{defn}

The writhe of the trefoil diagram, $K$ in Figure  ~\ref{trefoil}  is $+3$ so then its normalized bracket polynomial is given as,
\begin{eqnarray*}
X_{K}(A) & = & (-A)^{-3(3)}(A^{-7}-A^{-3}-A^{5})\\
      & = & -A^{-16}+A^{-12}+A^{-4}
\end{eqnarray*}

The Jones polynomial of the right-handed trefoil is therefore
\begin{equation}
V_{K}(t)=X_{K}(t^{-\frac{1}{4}})= -t^{4}+t^{3}+t
\end{equation}

\subsection{Calculating the Bracket Polynomial of Double Stranded Knots and Links.}

In this section, we apply the basic bracket expansion formula to create a switching formula for calculating double stranded links and doubles of knots.(refer to \cite{b6,a12})

By expanding on all four vertices, one finds the following formula for the bracket invariant;

\begin{eqnarray*}
\left \langle \raisebox{-0.6 em}{\includegraphics[scale=0.2]{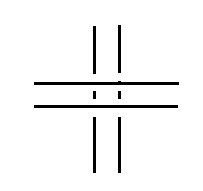}} \right \rangle & = & A^{4} \left \langle \raisebox{-0.6 em}{\includegraphics[scale=0.2]{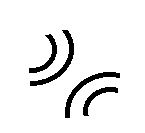}} \right \rangle + A^{-4}\left \langle \raisebox{-0.6 em}{\includegraphics[scale=0.2]{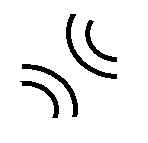}} \right \rangle + A^{2}\left[\left \langle \raisebox{-0.6 em}{\includegraphics[scale=0.2]{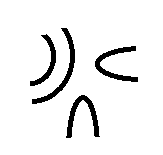}} \right \rangle + \left \langle \raisebox{-0.6 em}{\includegraphics[scale=0.2]{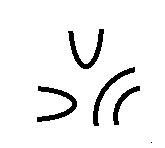}} \right \rangle \right] + A^{-2}\left[\left \langle \raisebox{-0.6 em}{\includegraphics[scale=0.2]{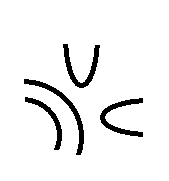}} \right \rangle + \left \langle \raisebox{-0.6 em}{\includegraphics[scale=0.2]{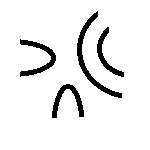}}\right \rangle \right] \\
 & + & (A^{2}+A^{-2})\left \langle \raisebox{-0.6 em}{\includegraphics[scale=0.2]{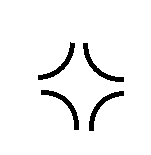}} \right \rangle + A^{0}\left[\left \langle \raisebox{-0.6 em}{\includegraphics[scale=0.2]{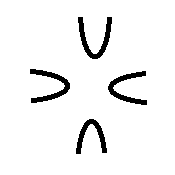}} \right \rangle + \left \langle \raisebox{-0.6 em}{\includegraphics[scale=0.2]{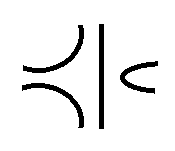}} \right \rangle + \left \langle \raisebox{-0.6 em}{\includegraphics[scale=0.2]{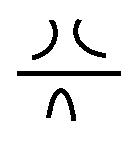}}\right \rangle + \left \langle \raisebox{-0.6 em}{\includegraphics[scale=0.2]{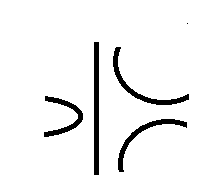}} \right \rangle + \left \langle \raisebox{-0.6 em}{\includegraphics[scale=0.2]{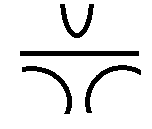}} \right \rangle \right]
\end{eqnarray*}

\begin{eqnarray*}
\left \langle \raisebox{-0.6 em}{\includegraphics[scale=0.2]{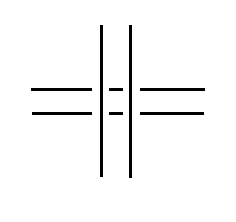}} \right \rangle & = & A^{4} \left \langle \raisebox{-0.6 em}{\includegraphics[scale=0.2]{s10.PNG}} \right \rangle + A^{-4}\left \langle \raisebox{-0.6 em}{\includegraphics[scale=0.2]{s6.PNG}} \right \rangle + A^{2}\left[\left \langle \raisebox{-0.6 em}{\includegraphics[scale=0.2]{s18.PNG}} \right \rangle + \left \langle \raisebox{-0.6 em}{\includegraphics[scale=0.2]{s21.PNG}} \right \rangle \right] + A^{-2}\left[\left \langle \raisebox{-0.6 em}{\includegraphics[scale=0.2]{s7.PNG}} \right \rangle + \left \langle \raisebox{-0.6 em}{\includegraphics[scale=0.2]{s11.PNG}} \right \rangle\right] \\
 & + & (A^{2}+A^{-2})\left \langle \raisebox{-0.6 em}{\includegraphics[scale=0.2]{s24.PNG}} \right \rangle + A^{0}\left[\left \langle \raisebox{-0.6 em}{\includegraphics[scale=0.2]{s8.PNG}} \right \rangle + \left \langle \raisebox{-0.6 em}{\includegraphics[scale=0.2]{s19.PNG}} \right \rangle + \left \langle \raisebox{-0.6 em}{\includegraphics[scale=0.2]{s12.PNG}} \right \rangle + \left \langle \raisebox{-0.6 em}{\includegraphics[scale=0.2]{s22.PNG}} \right \rangle + \left \langle \raisebox{-0.6 em}{\includegraphics[scale=0.2]{s25.PNG}}\right \rangle \right]
\end{eqnarray*}

The above two expansion formula immediately leads to the switching identity;

\begin{eqnarray*}
\left \langle  \raisebox{-0.6 em}{\includegraphics[scale=0.2]{t50.PNG}}\right \rangle - \left \langle  \raisebox{-0.6 em}{\includegraphics[scale=0.2]{t51.PNG}} \right \rangle & = & (A^{4}-A^{-4})\left[\left \langle  \raisebox{-0.6 em}{\includegraphics[scale=0.2]{s6.PNG}} \right \rangle -\left \langle \raisebox{-0.6 em}{\includegraphics[scale=0.2]{s10.PNG}} \right \rangle \right]+ (A^{2}-A^{-2})  [\left \langle \raisebox{-0.6 em}{\includegraphics[scale=0.2]{s7.PNG}} \right \rangle + \left \langle \raisebox{-0.6 em}{\includegraphics[scale=0.2]{s11.PNG}} \right \rangle\\
                                                    & - & \left \langle \raisebox{-0.6 em}{\includegraphics[scale=0.2]{s18.PNG}}\right \rangle - \left \langle \raisebox{-0.6 em}{\includegraphics[scale=0.2]{s21.PNG}} \right \rangle ]
\end{eqnarray*}

This identity will help us to compute the bracket polynomial and eventually the Jones polynomial of a double stranded link.

In other to do the calculation with much ease, we adapt a short hand notation of the above formula.We have that,
\begin{equation}
\left \langle \raisebox{-0.4 em}{\includegraphics[scale=0.3]{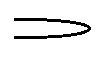}} \right \rangle = \left \langle \left \langle \raisebox{-0.4 em}{\includegraphics[scale=0.4]{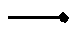}} \right \rangle \right \rangle
\end{equation}
and
\begin{equation}
\left \langle \raisebox{-0.6 em}{\includegraphics[scale=0.3]{s24.PNG}}\right \rangle =  \left \langle \left \langle \raisebox{-0.6 em}{\includegraphics[scale=0.3]{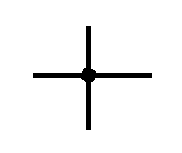}}\right \rangle \right \rangle
\end{equation}
while
\begin{equation}
\left \langle \raisebox{-0.6 em}{\includegraphics[scale=0.2]{s8.PNG}} \right \rangle =  \left \langle \left \langle \raisebox{-0.6 em}{\includegraphics[scale=0.3]{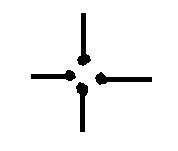}} \right \rangle \right \rangle
\end{equation}

Our expansion formula will thus be written as;
\begin{eqnarray*}
 \left \langle \left \langle \raisebox{-0.6 em}{\includegraphics[scale=0.2]{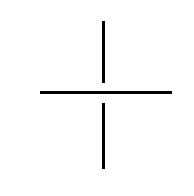}}\right \rangle \right \rangle & = & A^{4}  \left \langle \left \langle \raisebox{-0.6 em}{\includegraphics[scale=0.3]{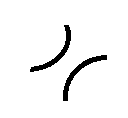}} \right \rangle \right \rangle + A^{-4} \left \langle \left \langle \raisebox{-0.6 em}{\includegraphics[scale=0.3]{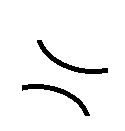}} \right \rangle \right \rangle + A^{2} \left[ \left \langle \left \langle \raisebox{-0.6 em}{\includegraphics[scale=0.3]{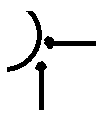}} \right \rangle \right \rangle +  \left \langle \left \langle \raisebox{-0.6 em}{\includegraphics[scale=0.3]{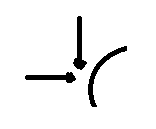}} \right \rangle \right \rangle \right]\\
                        & + & A^{-2}\left[ \left \langle \left \langle \raisebox{-0.6 em}{\includegraphics[scale=0.3]{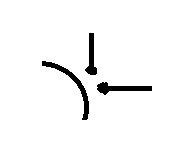}} \right \rangle \right \rangle +  \left \langle \left \langle \raisebox{-0.6 em}{\includegraphics[scale=0.3]{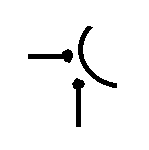}} \right \rangle \right \rangle \right] + (A^{2}+A^{-2}) \left \langle \left \langle \raisebox{-0.6 em}{\includegraphics[scale=0.3]{s68.PNG}} \right \rangle \right \rangle \\
                        & + & A^{0}\left[ \left \langle \left \langle \raisebox{-0.6 em}{\includegraphics[scale=0.25]{s16.PNG}} \right \rangle \right \rangle +  \left \langle \left \langle \raisebox{-0.6 em}{\includegraphics[scale=0.25]{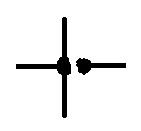}} \right \rangle \right \rangle +  \left \langle \left \langle \raisebox{-0.6 em}{\includegraphics[scale=0.25]{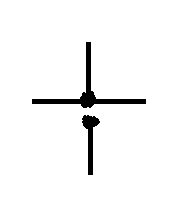}} \right \rangle \right \rangle +  \left \langle \left \langle \raisebox{-0.6 em}{\includegraphics[scale=0.25]{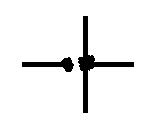}} \right \rangle \right \rangle +  \left \langle \left \langle \raisebox{-0.6 em}{\includegraphics[scale=0.25]{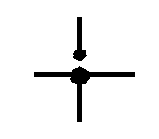}} \right \rangle \right \rangle \right]
\end{eqnarray*}

Noting for example that, 
\begin{equation}
 \left \langle \left \langle \raisebox{-0.6 em}{\includegraphics[scale=0.25]{s20.PNG}} \right \rangle \right \rangle = \left \langle \raisebox{-0.6 em}{\includegraphics[scale=0.25]{s19.PNG}} \right \rangle 
\end{equation}

\begin{eqnarray*}
 \left \langle \left \langle \raisebox{-0.6 em}{\includegraphics[scale=0.2]{t52.PNG}} \right \rangle \right \rangle -  \left \langle \left \langle \raisebox{-0.6 em}{\includegraphics[scale=0.2]{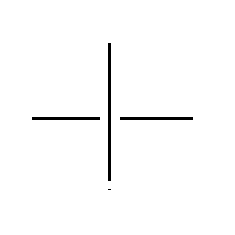}} \right \rangle \right \rangle & = & (A^{4}-A^{-4})\left[ \left \langle \left \langle \raisebox{-0.6 em}{\includegraphics[scale=0.2]{s1.PNG}} \right \rangle \right \rangle -  \left \langle \left \langle \raisebox{-0.6 em}{\includegraphics[scale=0.2]{s2.PNG}} \right \rangle \right \rangle \right]\\
                                       & + & (A^{2}-A^{-2})\left[ \left \langle \left \langle \raisebox{-0.6 em}{\includegraphics[scale=0.2]{s3.PNG}}\right \rangle \right \rangle +  \left \langle \left \langle \raisebox{-0.6 em}{\includegraphics[scale=0.2]{s5.PNG}} \right \rangle \right \rangle -  \left \langle \left \langle \raisebox{-0.6 em}{\includegraphics[scale=0.2]{s9.PNG}} \right \rangle \right \rangle -  \left \langle \left \langle \raisebox{-0.6 em}{\includegraphics[scale=0.2]{s4.PNG}} \right \rangle \right \rangle \right]
\end{eqnarray*}

Below are also some rules that can facilitate the calculation.
\begin{equation}
\left \langle \left \langle \raisebox{-0.3 em}{\includegraphics[scale=0.6]{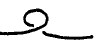}} \right \rangle \right \rangle = A^{8} \left \langle \left \langle \raisebox{-0.3 em}{\includegraphics[scale=0.4]{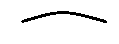}} \right \rangle \right \rangle + (A^{6}-A^{2}) \left \langle \left \langle \raisebox{-0.3 em}{\includegraphics[scale=0.4]{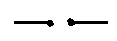}} \right \rangle \right \rangle
\end{equation}
\begin{equation}
\left \langle \left \langle \raisebox{-0.6 em}{\includegraphics[scale=0.6]{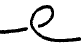}} \right \rangle \right \rangle = A^{-8} \left \langle \left \langle \raisebox{-0.3 em}{\includegraphics[scale=0.4]{s69.PNG}} \right \rangle \right \rangle + (A^{-6}-A^{-2}) \left \langle \left \langle \raisebox{-0.3 em}{\includegraphics[scale=0.4]{s67.PNG}} \right \rangle \right \rangle
\end{equation}
\begin{equation}
 \left \langle \left \langle \raisebox{-0.3 em}{\includegraphics[scale=0.6]{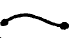}} \right \rangle \right \rangle = 1
\end{equation}
\begin{equation}
\left \langle \left \langle \raisebox{-0.6 em}{\includegraphics[scale=0.4]{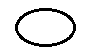}} \right \rangle \right \rangle=\delta
\end{equation}
\begin{equation}
\left \langle \left \langle \raisebox{-0.3 em}{\includegraphics[scale=0.3]{s64.PNG}} \raisebox{-0.6 em}{\includegraphics[scale=0.6]{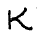}} \right \rangle \right \rangle = \delta \left \langle \left \langle \raisebox{-0.6 em}{\includegraphics[scale=0.6]{s66.PNG}} \right \rangle \right \rangle
\end{equation}
\begin{equation}
 \left \langle \left \langle \raisebox{-0.6 em}{\includegraphics[scale=0.6]{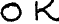}} \right \rangle \right \rangle = \delta^{2} \left \langle \left \langle \raisebox{-0.6 em}{\includegraphics[scale=0.6]{s66.PNG}} \right \rangle \right \rangle
\end{equation}

\begin{exa}[]
\label{exa:dblhopf}
 We wish to calculate the bracket polynomial of the 2-strand diagram of the standard Hopf link.
\end{exa}

Let us note that ;
\begin{equation}
\left \langle \left \langle \raisebox{-0.5 em}{\includegraphics[scale=0.6]{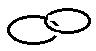}} \right \rangle \right \rangle = \left \langle \left \langle \raisebox{-0.2 em}{\includegraphics[scale=0.6]{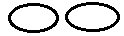}} \right \rangle \right \rangle = \left \langle  \raisebox{-0.2 em}{\includegraphics[scale=0.6]{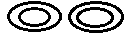}} \right \rangle = \delta^{3}
\end{equation}

Now applying the switching formula we have,
\begin{eqnarray*}
\left \langle \left \langle \includegraphics[scale=0.4]{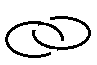} \right \rangle \right \rangle - \left \langle \left \langle \includegraphics[scale=0.4]{s59.PNG} \right \rangle \right \rangle & = & (A^{4}-A^{-4})\left[\left \langle \left \langle \includegraphics[scale=0.6]{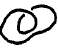} \right \rangle \right \rangle-\left \langle \left \langle \includegraphics[scale=0.6]{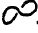} \right \rangle \right \rangle \right] \\
                      & + & (A^{2}-A^{-2})\left[\left \langle \left \langle \raisebox{-0.5 em}{\includegraphics[scale=0.6]{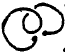}} \right \rangle \right \rangle + \left \langle \left \langle \raisebox{-0.5 em}{\includegraphics[scale=0.6]{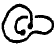}} \right \rangle \right \rangle - \left \langle \left \langle \raisebox{-0.5 em}{\includegraphics[scale=0.6]{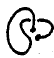}} \right \rangle \right \rangle - \left \langle \left \langle \raisebox{-0.5 em}{\includegraphics[scale=0.6]{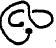}} \right \rangle \right \rangle\right]
\end{eqnarray*}

Therefore,
\begin{equation}
  \left \langle \left \langle \includegraphics[scale=0.6]{s85.PNG} \right \rangle \right \rangle = \delta^{3} + (A^{4}-A^{-4})\left[ \left \langle \left \langle \raisebox{-0.5 em}{\includegraphics[scale=0.6]{s56.PNG}} \right \rangle \right \rangle - \left \langle \left \langle \raisebox{-0.5 em}{\includegraphics[scale=0.6]{s51.PNG}} \right \rangle \right \rangle \right]
\end{equation}

We have that
\begin{eqnarray*}
  \left \langle \left \langle \raisebox{-0.5 em}{\includegraphics[scale=0.5]{s56.PNG}} \right \rangle \right \rangle & = & A^{8}  \left \langle \left \langle \includegraphics[scale=0.3]{s70.PNG} \right \rangle \right \rangle + (A^{6}-A^{2})  \left \langle \left \langle \includegraphics[scale=0.3]{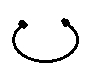} \right \rangle \right \rangle \\
                    & = & A^{8}(-A^{2}-A^{-2})+ A^{6}-A^{2}\\
                    & = & -A^{10}-A^{6}+A^{6}-A^{2}\\
                    & = & -A^{10}-A^{2}
 \end{eqnarray*}
 
 and
 \begin{eqnarray*}
  \left \langle \left \langle \raisebox{-0.5 em}{\includegraphics[scale=0.5]{s51.PNG}} \right \rangle \right \rangle & = & A^{-8}  \left \langle \left \langle \includegraphics[scale=0.3]{s70.PNG} \right \rangle \right \rangle + (A^{-6}-A^{-2})  \left \langle \left \langle \includegraphics[scale=0.3]{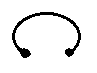} \right \rangle \right \rangle \\
                    & = & A^{-8}(-A^{2}-A^{-2})+ A^{-6}-A^{-2}\\
                    & = & -A^{-6}-A^{-10}+A^{-6}-A^{-2}\\
                    & = & -A^{-10}-A^{-2}
 \end{eqnarray*}
 
So then
\begin{eqnarray*}
  \left \langle \left \langle \raisebox{-0.5 em}{\includegraphics[scale=0.3]{s85.PNG}} \right \rangle \right \rangle & = & -(A^{6}+3A^{2}+3A^{-2}+A^{-6}) + (A^{4}-A^{-4})(-A^{10}-A^{2}+A^{-10}+A^{-2})\\
                      & = & -(A^{6}+3A^{2}+3A^{-2}+A^{-6}) + (-A^{14}+A^{2}+A^{-2}-A^{-14})\\
                      & = & -(A^{-14}+A^{-6}+2A^{-2}+2A^{2}+A^{6}+A^{14})
 \end{eqnarray*}
 \begin{equation}
 \left \langle \raisebox{-0.5 em}{\includegraphics[scale=0.3]{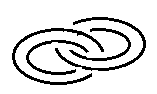}} \right \rangle = -(A^{-14}+A^{-6}+2A^{-2}+2A^{2}+A^{6}+A^{14})
 \end{equation}
 
 \section{The Bracket Polynomial of Tangles}
 
 \subsection{Description and Notations}
 
 \begin{defn}[]
\label{defn:tangle}
A \textbf{tangle} in a knot or link projection is a region in the projection plane surrounded by a circle such that the knot or link crosses the circle exactly four times.
\end{defn}

We will always think of the four parts where the knot or link crosses the circle as occurring in the four compass direction; North-West, North-East, South-West and South-East.

\begin{figure}[!h]
\centering 
\includegraphics[scale=0.6]{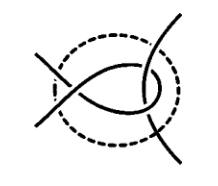}
\includegraphics[scale=0.6]{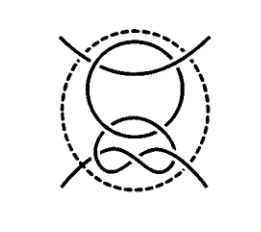}
\caption{Examples of tangle}
\label{Egtangles} 
\end{figure}

Let $T$ be a tangle in a knot or link projection, then by convention we denote $-T$ to be its reflection(or mirror image). For a twist of $(n)$ crossings proceeding from west to east, we use the twist that is right-handed for $n>0$ and left-handed for $n<0$. 

\begin{figure}[!ht]
    \subfloat[$3$\label{subfig-1:dummy}]{%
      \includegraphics[scale=1.0]{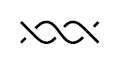}
    }
    \hfill
    \subfloat[$-3$ \label{subfig-2:dummy}]{%
      \includegraphics[scale=1.0]{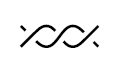}
    }
    \caption{}
    \label{intangles}
  \end{figure}
  
Given a tangle $T(\raisebox{-0.6 em}{\includegraphics[scale=0.4]{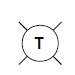}})$, we obtain the \textbf{numerator}, $T^{N}$ by attaching the (top) $NW$ and $NE$ edges of $T$ together and attaching the (bottom) $SW$ and $SE$ edges
together. Also the \textbf{denominator}, $T^{D}$ is obtained by attaching the (left side) $NW$ and $SW$ edges together and attaching the (right side) $NE$ and $SE$ edges together.

Given tangles $T(\raisebox{-0.6 em}{\includegraphics[scale=0.45]{t6.PNG}})$ and $U(\raisebox{-0.8 em}{\includegraphics[scale=0.3]{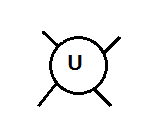}})$, one defines the \textbf{sum}, denoted $T + U$ by placing the diagram for $U$ to the right of the diagram for $T$ and attaching the $NE$ edge of $T$ to the $NW$ edge of $U$, and the $SE$ edge of $T$ to the $SW$ edge of $S$. The resulting tangle $T + U$ has exterior edges corresponding to the $NW$ and $SW$ edges of $T$ and the $NE$ and $SE$ edges of $U$. We shall also have occasion to consider the \textbf{"vertical sum"} of two tangles say $T$ and $U$ by placing the diagram for $U$ below the diagram for $T$ and attaching the $SE$ edge of $T$ to the $NE$ edge of $U$, and the $SW$ edge of $T$ to the $NW$ edge of $U$. The resulting tangle $T \ast U$ has exterior edges corresponding to the $NW$ and $NE$ edges of $T$ and the $SW$ and $SE$ edges of $U$.

\begin{figure}[!h]
\centering 
\includegraphics[scale=0.5]{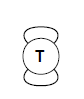}
\caption{Numerator tangle of $T$, $T^{N}$}
\label{numeartot} 
\includegraphics[scale=0.5]{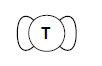}
\caption{Denominator tangle of $T$, $T^{D}$}
\label{den} 
\includegraphics[scale=0.5]{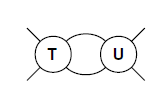}
\caption{Tangle $T+S$}
\label{sum}
\includegraphics[scale=0.5]{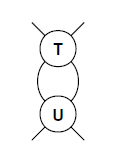}
\caption{Tangle $T \ast S$}
\label{vert sum}
\end{figure} 

\newpage
We note that the "integer" tangle $-n$ is indeed an additive inverse of the tangle $n$, in that their sum is $0= \raisebox{-0.6 em}{\includegraphics[scale=0.3]{t21.PNG}}$.

We shall denote by $T^{-1}$, the reflection of $T$ in a $NW-SE$ axis. In other words rotating all of it's edges through $90^{\circ}$ counter-clockwise and finding the mirror image. We define the product of two tangles, $T.U$ as the tangle sum $T^{-1}+U$. Note that $T.0 = T^{-1}$. Traditionally for $n \in \mathbb{Z}$, $n \neq 0$, the tangle $n^{-1}$ is denoted $\frac{1}{n}$ and $0^{-1}=\raisebox{-0.5 em}{\includegraphics[scale=0.3]{t24.PNG}}$ is denoted $\infty$.

\begin{figure}[!ht]
    \subfloat[$3^{-1}$ \label{subfig-1:dummy}]{%
      \includegraphics[scale=0.55]{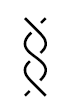}
    }
    \hfill
    \subfloat[$3 \cdot 2$\label{subfig-2:dummy}]{%
      \includegraphics[scale=0.55]{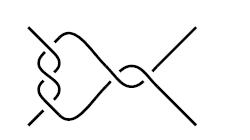}
    }
    \caption{}
    \label{egtangles}
  \end{figure}
  
 Regarding iterated "product", we follow the standard convention that $T \cdot U \cdot V$ is to be interpreted as $(T \cdot U) \cdot V$, i.e 
 \begin{eqnarray*}
 T \cdot U \cdot V & = & (T \cdot U) \cdot V \\
       & = & (T^{-1}+U) \cdot V \\
       & = & [(T^{-1}+U)]^{-1} + V
 \end{eqnarray*}

The class of examples that we consider in this paper are each denoted by $H(T,U)$, where $T$ and $U$ are each tangles and $H(T,U)$ is a satellite of the Hopf link that conforms to the pattern shown in Figure ~\ref{sathopf} below.

\begin{figure}[!h]
\centering 
\includegraphics[scale=0.55]{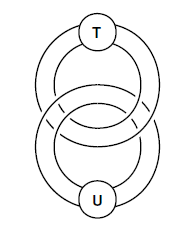}
\caption{$H(T,U)$}
\label{sathopf} 
\end{figure}

\subsection{The Bracket Polynomial of $H(T,U)$ }

Given a tangle, $T$, the bracket expansion formula; $\left \langle \raisebox{-0.8 em}{\includegraphics[scale=0.3]{t55.PNG}} \right \rangle = A\left \langle \raisebox{-0.6 em}{\includegraphics[scale=0.3]{t24.PNG}} \right \rangle + A^{-1}\left \langle \raisebox{-0.6 em}{\includegraphics[scale=0.3]{t21.PNG}} \right \rangle$ together with the rule, $\left \langle L \sqcup O \right \rangle = \delta \left \langle L \right \rangle $(applicable also to link diagrams) allows us to express the symbol $\langle T \rangle$ as a formal linear combination. We have that $\langle T \rangle = f(T)\langle 0 \rangle+g(T)\langle \infty \rangle$, where $\langle 0 \rangle$,$\langle \infty \rangle$ are regarded as primitive objects and the coefficient $f(T)$,$g(T)$ are in the ring $\mathbb{Z}[A,A^{-1}]$.

We define the bracket vector of $T$ to be the ordered pair $(f(T),g(T))$ and denote it by $br(T)$. For example $br(1)=(A,A^{-1})$. Where appropriate, we shall consider $br(T)$ as the column vector $\begin{pmatrix}
f(T) \\ 
g(T)
\end{pmatrix}$.
\begin{exa}[]
\label{exa:numerator}
Let us consider the tangle $2$ and obtain its bracket vector.
\end{exa}

Traditionally we know that the tangle $2$ is represented by the diagram, \raisebox{-0.4 em}{\includegraphics[scale=0.3]{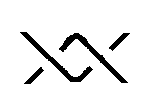}} and we can compute its bracket polynomial as follows,
\begin{eqnarray*}
\left \langle \raisebox{-0.2 em}{\includegraphics[scale=0.3]{c1.PNG}} \right \rangle & = & A \left \langle \raisebox{-0.2 em}{\includegraphics[scale=0.3]{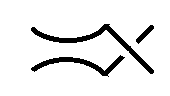}} \right \rangle + A^{-1} \left \langle \raisebox{-0.2 em}{\includegraphics[scale=0.3]{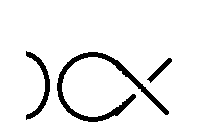}} \right \rangle\\
                             & = & A\left[ A \left \langle \raisebox{-0.2 em}{\includegraphics[scale=0.3]{t21.PNG}} \right \rangle + A^{-1} \left \langle \raisebox{-0.4 em}{\includegraphics[scale=0.3]{t24.PNG}} \right \rangle \right] + A^{-1}\left[ A \left \langle \raisebox{-0.4 em}{\includegraphics[scale=0.3]{t24.PNG}} \right \rangle + A^{-1}\left \langle \raisebox{-0.2 em}{\includegraphics[scale=0.3]{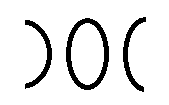}} \right \rangle \right]\\
                             & = & A^{2} \left \langle \raisebox{-0.2 em}{\includegraphics[scale=0.3]{t21.PNG}} \right \rangle + (2+A^{-2}(-A^{-2}-A^{2}))\left \langle \raisebox{-0.4 em}{\includegraphics[scale=0.3]{t24.PNG}} \right \rangle\\
                             & = & A^{2} \left \langle \raisebox{-0.2 em}{\includegraphics[scale=0.3]{t21.PNG}} \right \rangle + (2-A^{-4}-1)\left \langle \raisebox{-0.4 em}{\includegraphics[scale=0.3]{t24.PNG}} \right \rangle\\
                             & = & A^{2} \left \langle \raisebox{-0.2 em}{\includegraphics[scale=0.3]{t21.PNG}} \right \rangle + (1-A^{-4})\left \langle \raisebox{-0.4 em}{\includegraphics[scale=0.3]{t24.PNG}} \right \rangle\\
                              & = & A^{2} \langle 0 \rangle + (1-A^{-4})\langle \infty \rangle
\end{eqnarray*}
In this example we have that $f(T)= A^{2}$ and $g(T)=(1-A^{-4})$. Our bracket vector is therefore given as $br(2) = \begin{pmatrix}
A^{2} \\ 
1-A^{-4}
\end{pmatrix}$.

\begin{prop}[\cite{a2}]
\label{prop:numerator}

\end{prop}
\begin{enumerate}
\item $\begin{pmatrix}
\langle T^{N} \rangle \\ 
\langle T^{D} \rangle
\end{pmatrix} = \begin{pmatrix}
\delta & 1 \\ 
1 & \delta
\end{pmatrix} br(T)$
\item $br(T+U)= \begin{pmatrix}
f(U) & 0 \\ 
g(U) & f(U)+\delta g(U)
\end{pmatrix} br(T).$

$br(T \ast U)= \begin{pmatrix}
\delta f(U) +g(U) & f(U) \\ 
0 &  g(U)
\end{pmatrix} br(T).$
\end{enumerate} 

\begin{proof}
\begin{enumerate}

\item From the above expression of the bracket polynomial of $T$, we know that;

$\langle T \rangle = f(T)\langle 0 \rangle+g(T)\langle \infty \rangle$, where $0=\raisebox{-0.6 em}{\includegraphics[scale=0.3]{t21.PNG}} $, $\infty =\raisebox{-0.6 em}{\includegraphics[scale=0.3]{t24.PNG}}$

We also note that $T^{N}=\raisebox{-0.6 em}{\includegraphics[scale=0.3]{t7.PNG}}$ and $T^{D}= \raisebox{-0.5 em}{\includegraphics[scale=0.3]{t8.PNG}}$.
\begin{eqnarray*}
\langle T^{N} \rangle = \left \langle \raisebox{-0.6 em}{\includegraphics[scale=0.3]{t7.PNG}} \right \rangle & = & f(T)\left \langle \raisebox{-0.8 em}{\includegraphics[scale=0.3]{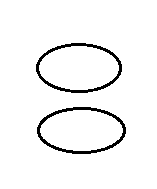}} \right \rangle + g(T)\left \langle \raisebox{-0.8 em}{\includegraphics[scale=0.3]{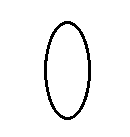}} \right \rangle \\
                     & = & \delta f(T)+g(T)
\end{eqnarray*}
Also,
\begin{eqnarray*}
\langle T^{D} \rangle = \left \langle \raisebox{-0.4 em}{\includegraphics[scale=0.3]{t8.PNG}} \right \rangle & = & f(T)\left \langle \raisebox{-0.6 em}{\includegraphics[scale=0.3]{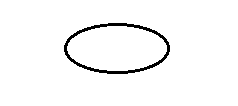}} \right \rangle + g(T)\left \langle \raisebox{-0.6 em}{\includegraphics[scale=0.3]{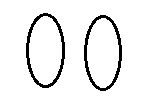}} \right \rangle \\
                     & = & f(T)+ \delta  g(T)
\end{eqnarray*}
therefore,
\begin{eqnarray*}
 \begin{pmatrix}
\langle T^{N} \rangle \\ 
\langle T^{D} \rangle
\end{pmatrix} = \begin{pmatrix}
\delta & 1 \\ 
1 & \delta
\end{pmatrix} 
\begin{pmatrix}
f(T) \\ 
g(T)
\end{pmatrix} =  \begin{pmatrix}
\delta & 1 \\ 
1 & \delta
\end{pmatrix} br(T)
\end{eqnarray*} 

\item We know that $T+U$ is denoted as \raisebox{-0.6 em}{\includegraphics[scale=0.4]{t9.PNG}}.
\begin{eqnarray*}
\langle T+U \rangle = \left \langle \raisebox{-0.6 em}{\includegraphics[scale=0.4]{t9.PNG}} \right \rangle & = & f(T)\left \langle \raisebox{-0.6 em}{\includegraphics[scale=0.25]{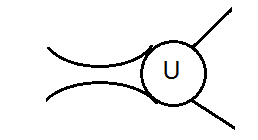}} \right \rangle + g(T)\left \langle \raisebox{-0.6 em}{\includegraphics[scale=0.25]{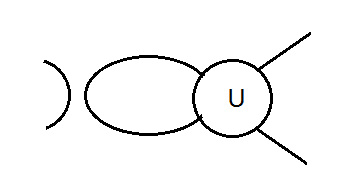}} \right \rangle \\
              & = & f(T)[f(U) \left \langle \raisebox{-0.8 em}{\includegraphics[scale=0.25]{t21.PNG}} \right \rangle+ g(U)\left \langle \raisebox{-0.8 em}{\includegraphics[scale=0.25]{t24.PNG}} \right \rangle]+ g(T)[f(U)\left \langle \raisebox{-0.8 em}{\includegraphics[scale=0.25]{t24.PNG}} \right \rangle\\
              & + & g(U)\left \langle \raisebox{-0.6 em}{\includegraphics[scale=0.25]{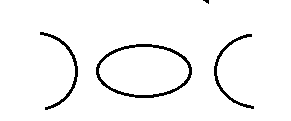}} \right \rangle] \\
              & = & f(T)[f(U) \left \langle \raisebox{-0.8 em}{\includegraphics[scale=0.25]{t21.PNG}} \right \rangle+ g(U)\left \langle \raisebox{-0.8 em}{\includegraphics[scale=0.25]{t24.PNG}} \right \rangle] + g(T)[f(U)\left \langle \raisebox{-0.8 em}{\includegraphics[scale=0.25]{t24.PNG}} \right \rangle\\
              & + & \delta g(U)\left \langle \raisebox{-0.8 em}{\includegraphics[scale=0.25]{t24.PNG}} \right \rangle]  \\
              & = & f(T)f(U)\left \langle \raisebox{-0.8 em}{\includegraphics[scale=0.2]{t21.PNG}} \right \rangle + [f(T)g(U)+ g(T)f(U)+ \delta g(T)g(U)]\left \langle \raisebox{-0.8 em}{\includegraphics[scale=0.2]{t24.PNG}} \right \rangle
\end{eqnarray*}
$\Longrightarrow f(T+U)=f(T)f(U)$ and $ g(T+U)=f(T)g(U)+ g(T)f(U)+ \delta g(T)g(U)$

therefore,
\begin{eqnarray*}
 br(T+U)= \begin{pmatrix}
 f(T)f(U) \\ 
f(T)g(U)+ g(T)f(U)+ \delta g(T)g(U)
\end{pmatrix} 
& = & 
\begin{pmatrix}
f(U) & 0 \\ 
g(U) & f(U)+ \delta g(U)
\end{pmatrix} 
\begin{pmatrix}
f(T) \\ 
g(T)
\end{pmatrix} \\
& = & 
\begin{pmatrix}
f(U) & 0 \\ 
g(U) & f(U)+ \delta g(U)
\end{pmatrix} br(T)
\end{eqnarray*}

Also, we know that $T \ast U$ is denoted as $\raisebox{-0.6 em}{\includegraphics[scale=0.4]{t10.PNG}}$
\begin{eqnarray*}
\langle T \ast U \rangle = \left \langle \raisebox{-1.2 em}{\includegraphics[scale=0.4]{t10.PNG}} \right \rangle & = & f(T)\left \langle \raisebox{-1.4 em}{\includegraphics[scale=0.2]{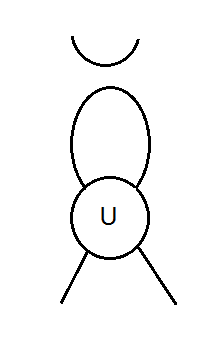}} \right \rangle + g(T)\left \langle \raisebox{-1.2 em}{\includegraphics[scale=0.25]{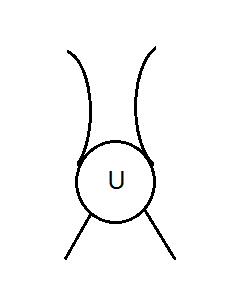}} \right \rangle \\
              & = & f(T)[f(U) \left \langle \raisebox{-1.6 em}{\includegraphics[scale=0.25]{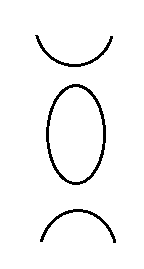}} \right \rangle+ g(U)\left \langle \raisebox{-0.8 em}{\includegraphics[scale=0.25]{t21.PNG}} \right \rangle]+ g(T)[f(U)\left \langle \raisebox{-0.8 em}{\includegraphics[scale=0.25]{t21.PNG}} \right \rangle \\
              & + & g(U)\left \langle \raisebox{-0.8 em}{\includegraphics[scale=0.25]{t24.PNG}} \right \rangle] \\
              & = & f(T)[\delta f(U) \left \langle \raisebox{-0.8 em}{\includegraphics[scale=0.25]{t21.PNG}} \right \rangle+ g(U)\left \langle \raisebox{-0.8 em}{\includegraphics[scale=0.25]{t21.PNG}} \right \rangle] + g(T)[f(U)\left \langle \raisebox{-0.8 em}{\includegraphics[scale=0.25]{t21.PNG}} \right \rangle\\ 
              & + & g(U)\left \langle \raisebox{-0.8 em}{\includegraphics[scale=0.25]{t24.PNG}} \right \rangle]  \\
              & = & [\delta f(T)f(U)+  f(T)g(U)+ g(T)f(U)]\left \langle \raisebox{-0.8 em}{\includegraphics[scale=0.25]{t21.PNG}} \right \rangle + g(T)g(U)\left \langle \raisebox{-0.8 em}{\includegraphics[scale=0.25]{t24.PNG}} \right \rangle
\end{eqnarray*}
$\Longrightarrow f(T \ast U)=\delta f(T)f(U)+  f(T)g(U)+ g(T)f(U) $ and $ g(T \ast U)= g(T)g(U))$

therefore,
\begin{eqnarray*}
 br(T \ast U)= \begin{pmatrix}
 \delta f(T)f(U)+  f(T)g(U)+ g(T)f(U) \\ 
 g(T)g(U)
\end{pmatrix} 
& = & 
\begin{pmatrix}
\delta f(U)+  g(U) & f(U) \\ 
0 & g(U)
\end{pmatrix} 
\begin{pmatrix}
f(T) \\ 
g(T)
\end{pmatrix} \\
& = & 
\begin{pmatrix}
\delta f(U)+  g(U) & f(U) \\ 
0 & g(U)
\end{pmatrix} br(T)
\end{eqnarray*}
\end{enumerate}
\end{proof}

Now considering the class of link under study,$H(T,U)$. We observe that if we take $(T,U)=(0,0)$ we obtain the $2$-strand parallel of the standard diagram of the Hopf link.
The bracket polynomial of that diagram was computed in Example ~\ref{exa:dblhopf} of the previous chapter to be; 
$$ -(A^{-14}+A^{-6}+2A^{-2}+2A^{2}+A^{6}+A^{14})$$

The choice $(T,U)=(0,\infty)$(or $(T,U)=(\infty, 0)$) yields a diagram with writhe 0 of the unlink with $3$ components and $(T,U)=(\infty,\infty)$ gives a diagram with writhe 0 of the unlink with $2$ components.
Therefore the bracket polynomial of $H(0,\infty)$, $H(\infty,\infty)$ are $\delta^{2}$,$\delta$ respectively.

Let us define;
\begin{equation}
h_{00}=\langle H(0,0) \rangle = -(A^{-14}+A^{-6}+2A^{-2}+2A^{2}+A^{6}+A^{14}) 
\end{equation}
\begin{equation}
h_{01}=h_{10}= \langle H(0,\infty) \rangle = \delta^{2}
\end{equation}
\begin{equation}
h_{11}= \langle H(\infty,\infty) \rangle = \delta
\end{equation}

and let $\mathcal{M}$ denote the matrix,
$\begin{pmatrix}
h_{00} & h_{01} \\ 
h_{10} & h_{11}
\end{pmatrix} $

\begin{prop}[\cite{a2}]
\label{prop:jwt}
$\langle H(T,U) \rangle=h_{00}f(T)f(U)+h_{01}[f(T)g(U)+g(T)f(U)]+h_{11}g(T)g(U)$

or in the matrix notation,
 
$\langle H(T,U) \rangle=br(T)^{t} \cdot \mathcal{M} \cdot br(U)$
\end{prop}
\begin{proof}
\begin{eqnarray*}
\langle H(T,U) \rangle & = &  \left \langle \raisebox{-1.5 em}{\includegraphics[scale=0.25]{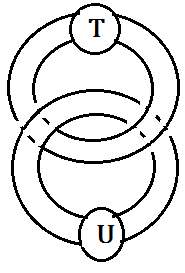}} \right \rangle \\
                       & = & f(T)\left \langle \raisebox{-1.6 em}{\includegraphics[scale=0.25]{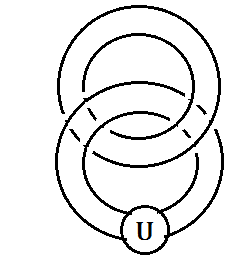}} \right \rangle + g(T)\left \langle \raisebox{-1.6 em}{\includegraphics[scale=0.25]{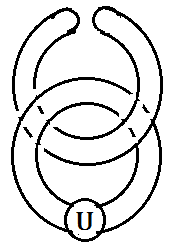}} \right \rangle\\
                       & = &f(T)\left[f(U)\left \langle \raisebox{-1.6 em}{\includegraphics[scale=0.25]{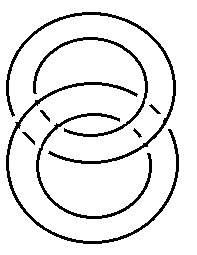}} \right \rangle + g(U) \left \langle \raisebox{-1.6 em}{\includegraphics[scale=0.25]{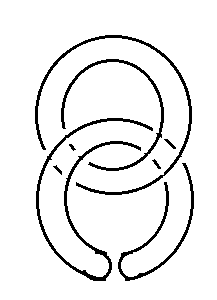}} \right \rangle \right]+ g(T) \left[f(U)\left \langle \raisebox{-1.6 em}{\includegraphics[scale=0.25]{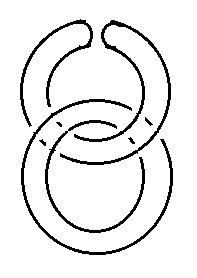}} \right \rangle + g(U)\left \langle \raisebox{-1.6 em}{\includegraphics[scale=0.25]{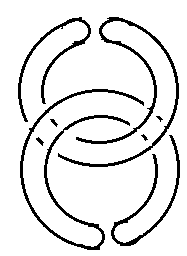}} \right \rangle \right]\\
                       & = & f(T)f(U)\langle H(0,0)\rangle +  f(T)g(U)\langle H(0,\infty)\rangle +  g(T)f(U)\langle H(\infty,0)\rangle +  g(T)g(U)\langle H(\infty,\infty)\rangle\\
                       & = &  h_{00}f(T)f(U) + h_{01}[f(T)g(U)+g(T)f(U)]+h_{11}g(T)g(U)
\end{eqnarray*}

We can now check that $br(T)^{t} \cdot \mathcal{M} \cdot br(U) = h_{00}f(T)f(U) + h_{01}(f(T)g(U)+g(T)f(U))+h_{11}g(T)g(U)$

\begin{eqnarray*}
br(T)^{t} \cdot \mathcal{M} \cdot br(U) & = &  br(T)^{t} \cdot (\mathcal{M} \cdot br(U))\\
                            & = & br(T)^{t} \cdot \left[ \begin{pmatrix}
h_{00} & h_{01} \\ 
h_{10} & h_{11}
\end{pmatrix} \cdot \begin{pmatrix}
f(U) \\ 
g(U)
\end{pmatrix} \right] \\
                  & = & \begin{pmatrix}
f(T) & g(T)
\end{pmatrix} \cdot \begin{pmatrix}
h_{00}f(U) + h_{01}g(U) \\ 
h_{10}f(U) + h_{11}g(U)
\end{pmatrix}  \\
                 & = & f(T) [h_{00}f(U) + h_{01}g(U)] +g(T)[h_{10}f(U) + h_{11}g(U)]\\
                 & = & h_{00}f(T)f(U) + h_{01}[f(T)g(U)+g(T)f(U)] +  h_{11}g(T)g(U)
\end{eqnarray*}
\end{proof}

\section{The Present Status of Links with Trivial Jones Polynomial}
\subsection{The Omega Operation.}
We apply a method based on the transformation $H(T,U) \longrightarrow H(T,U) ^{\omega}$ to construct links with trivial Jones polynomial. In this transformation, tangles $T$ and $U$ are cut out and re-glued by certain specific homeomorphisms of the tangle boundaries. The transformation $\omega$ has a property of preserving the Kauffman bracket polynomial. It is however, effective in generating examples, as a trivial link can be transformed to a prime link and repeated application yields an infinite sequence of inequivalent links.
\begin{defn}[\cite{a2}]
\label{defn:jwt}
Given a tangle $T$, $T^{\omega}$ denotes the tangle $(T+2).1.2$ and $T^{\bar{\omega}}$ the tangle $(T-2).(-1).(-2)$.
\end{defn}
\begin{figure}[!h]
\centering 
\includegraphics[scale=0.6]{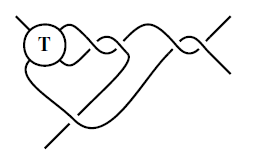}
\includegraphics[scale=0.6]{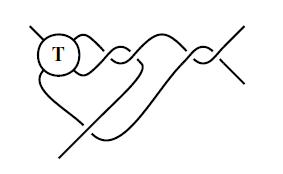}
\caption{$T^{\omega}$ and $T^{\bar{\omega}}$ }
\label{Omegaope} 
\end{figure}

Using the ''vertical sum'' operation introduced in chapter $3$, we can as well write,
\begin{equation}
T^{\omega}=((T+2)*1)+2
\end{equation}
\begin{equation}
T^{\bar{\omega}}=((T-2)*(-1))-2
\end{equation}
At this point we determine  the effect that the operation $\omega$ and $\bar{\omega}$ have on bracket vector of a tangle.

 From Proposition ~\ref{prop:numerator}  $2$ we have,
 \begin{eqnarray*}
br(T+1) & = & \begin{pmatrix}
A & 0 \\ 
A^{-1} & A + \delta A^{-1}
\end{pmatrix}br(T),\qquad \text{where} \quad \delta = -A^{2}-A^{-2}\\
        & = & \begin{pmatrix}
A & 0 \\ 
A^{-1} & -A^{-3}
\end{pmatrix}br(T) 
\end{eqnarray*}
and
\begin{eqnarray*}
br(T*1) & = & \begin{pmatrix}
\delta A + A^{-1}& A \\ 
0 & A^{-1}
\end{pmatrix}br(T),\qquad \text{where} \quad \delta = -A^{2}-A^{-2}\\
        & = & \begin{pmatrix}
-A^{3} & A \\ 
0 & A^{-1}
\end{pmatrix}br(T) 
\end{eqnarray*}

Similarly we have that,
\begin{equation}
br(T+(-1))=M_{+}^{-1} \cdot br(T)
\end{equation}
\begin{equation}
br(T*(-1))=M_{*}^{-1} \cdot br(T)
\end{equation}
where

$M_{+}=\begin{pmatrix}
A & 0 \\ 
A^{-1} & -A^{-3}
\end{pmatrix}$\quad , \quad $M_{*}=\begin{pmatrix}
-A^{3} & A \\ 
0 & A^{-1}
\end{pmatrix}$

In this context it is natural to introduce the $2 \times 2$ matrix,
\begin{equation} \label{eq:omega}
\Omega = M_{+}^{2}M_{*}M_{+}^{2}=\begin{pmatrix}
-A^{-1}+A^{3}-A^{7} & A^{-3} \\ 
-A^{-11}+2A^{-7}-2A^{-3}+2A-A^{5} & A^{-13}-A^{-9}+A^{-5}
\end{pmatrix} 
\end{equation}

\begin{prop}[\cite{a2}]
\label{prop:jwt}
\begin{enumerate}
\item $br(T^{\omega}) = \Omega \cdot br(T)$
\item $br(T^{\bar{\omega}}) = \Omega^{-1} \cdot br(T)$
\end{enumerate}
\end{prop}
\begin{proof}
\begin{eqnarray*}
br(T^{\omega}) & = & br(((T+2)*1)+2)\\
               & = & M_{+} \cdot br(((T+2)*1)+1)\\
               & = & M_{+}^{2} \cdot br((T+2)*1)\\
               & = & M_{+}^{2}M_{*} \cdot br(T+2)\\
               & = & M_{+}^{2}M_{*}M_{+} \cdot br(T+1)\\
               & = & M_{+}^{2}M_{*}M_{+}^{2} \cdot br(T)\\
               & = & \Omega \cdot br(T)
\end{eqnarray*}
\begin{eqnarray*}
br(T^{\bar{\omega}}) & = & br(((T-2)*(-1))-2)\\
                     & = & M_{+}^{-1} \cdot br(((T-2)*(-1))-1)\\
                     & = & M_{+}^{-2} \cdot br((T-2)*(-1))\\
                     & = & M_{+}^{-2}M_{*}^{-1} \cdot br(T-2)\\
                     & = & M_{+}^{-2}M_{*}^{-1}M_{+}^{-1} \cdot br(T-1)\\
                     & = & M_{+}^{-2}M_{*}^{-1}M_{+}^{-2} \cdot br(T)\\
                     & = & \Omega^{-1} \cdot br(T)
\end{eqnarray*}
\end{proof}

\begin{defn}[\cite{a2}]
\label{defn:jwt}
Given tangles $T$,$U$, $H(T,U)^{\omega}$ denotes the diagram $H(T^{\omega},U^{\bar{\omega}})$(i.e $H(T,U)^{\omega} = H(T^{\omega},U^{\bar{\omega}})$).

In this context we designate $H(T,U)^{\omega} = H(T^{\omega},U^{\bar{\omega}})$ as the diagram obtained by replacing $T$ by $T^{\omega}$ and $U$ by $U^{\bar{\omega}}$.

\end{defn}
\begin{thm}[\cite{a2}]
\label{thm:final}
Let $T$, $U$ be any tangles. Then the bracket polynomials of $H(T,U)$ and $H(T,U)^{\omega}$ are equal.
\end{thm}
\begin{proof}
\begin{eqnarray*}
\langle H(T^{\omega},U^{\bar{\omega}}) \rangle & = & br(T^{\omega})^{t} \cdot \mathcal{M} \cdot  br(U^{\bar{\omega}})\\
              & = & (\Omega \cdot br(T))^{t}\cdot \mathcal{M} \cdot \Omega^{-1} \cdot br(U)\\
              & = & (br(T))^{t} \cdot \Omega^{t} \cdot \mathcal{M} \cdot \Omega^{-1} \cdot br(U)
\end{eqnarray*}
We now need to verify that $\Omega^{t} \cdot \mathcal{M} \cdot \Omega^{-1} = \mathcal{M}$, where
\begin{equation}
\mathcal{M} = \begin{pmatrix}
-(A^{-14}+A^{-6}+2A^{-2}+2A^{2}+A^{6}+A^{14}) & A^{-4}+A^{4}+2 \\ 
A^{-4}+A^{4}+2 & -A^{-2}-A^{2}
\end{pmatrix}
\end{equation}

Considering the matrix $\Omega$ in Equation  ~\eqref{eq:omega}

\begin{eqnarray*}
\Omega^{t} = \begin{pmatrix}
-A^{-1}+A^{3}-A^{7} & -A^{-11}+2A^{-7}-2A^{-3}+2A-A^{5} \\ 
A^{-3} & A^{-13}-A^{-9}+A^{-5}
\end{pmatrix} 
\end{eqnarray*}
\begin{eqnarray*}
\Omega^{-1} & = & -\frac{1}{A^{-6}} \begin{pmatrix}
 A^{-13}-A^{-9}+A^{-5} & -A^{-3} \\ 
A^{-11}-2A^{-7}+2A^{-3}-2A+A^{5} & -A^{-1}+A^{3}-A^{7}
\end{pmatrix} \\
            & = & \begin{pmatrix}
-A^{-7}+A^{-3}-A^{1} & A^{3} \\ 
-A^{-5}+2A^{-1}-2A^{3}+2A^{7}-A^{11} & A^{5}-A^{9}+A^{13}
\end{pmatrix} 
\end{eqnarray*}

Now by direct matrix multiplication we have that,
\begin{equation}
\Omega^{t} \cdot \mathcal{M} = \begin{pmatrix}
-A^{-11}+3A^{-7}+A^{-3}+2A+A^{5}+2A^{13}-A^{17}+A^{21} & A^{-13}-A^{-9}-A^{-5}-A^{-1}-A^{3}-A^{11} \\ 
A^{-13}-A^{-9}-A^{-5}-A^{-1}-A^{3}-A^{11} & -A^{-15}+A^{-7}+A^{-3}+A
\end{pmatrix}
\end{equation}

We therefore have that,
\begin{equation}
(\Omega^{t} \cdot \mathcal{M}) \cdot \Omega^{-1} = \begin{pmatrix}
-(A^{-14}+A^{-6}+2A^{-2}+2A^{2}+A^{6}+A^{14}) & A^{-4}+A^{4}+2 \\ 
A^{-4}+A^{4}+2 & -A^{-2}-A^{2}
\end{pmatrix} = \mathcal{M}
\end{equation}
and hence
\begin{eqnarray*}
\langle H(T^{\omega},U^{\bar{\omega}}) \rangle & = &   (br(T))^{t} \cdot \Omega^{t} \cdot \mathcal{M} \cdot \Omega^{-1} \cdot br(U)\\
                                               & = & (br(T))^{t} \cdot \mathcal{M} \cdot br(U)\\
                                               & = & \langle H(T,U) \rangle
\end{eqnarray*}

\end{proof}

\subsection{Thistlethwaite's Example}

The figure  ~\ref{Thistlethwaite} below is a version of a link discovered by Morwen Thistlethwaite \cite{a3b} in December $2000$. This is an example of a link which is linked but whose linking cannot be detected by the Jones polynomial.This fact can be verified using a computer program or a polynomial invariant like the Kauffman polynomial.

\begin{figure}[!h]
\centering 
\includegraphics[scale=0.5]{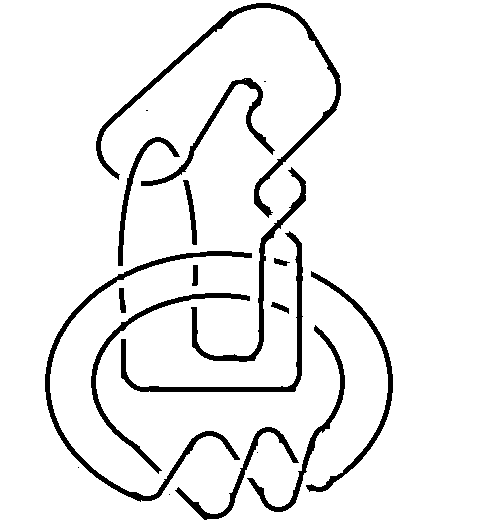}
\caption{Thistlethwaite's link}
\label{Thistlethwaite} 
\end{figure}

It interesting to note that we can verify that the Thistlethwaite's link fits into the structure we consider in this paper. This can be done by choosing tangle $T$ as $-1$ and tangle $U$ as $\infty + 2$ resulting in the figure ~\ref{t-unlinked}  below with writhe $-3$.

\begin{figure}[!h]
\centering 
\includegraphics[scale=0.5]{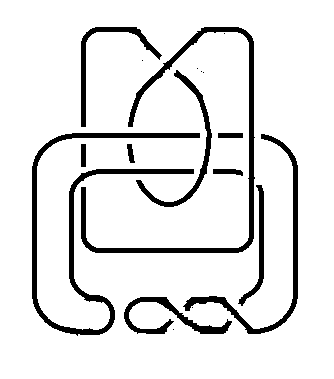}
\caption{Unlink with writhe $-3$}
\label{t-unlinked} 
\end{figure}

We perform the omega transformation on the tangles $T$ and $U$ and make use of regular isotopy as shown in ~\eqref{Tbar} and ~\eqref{Ubar}.

\begin{equation}\label{Tbar}
 \includegraphics[scale=0.2]{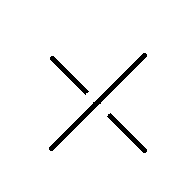}^{ \qquad \xrightarrow{\huge{\omega}}} \qquad \includegraphics[scale=0.2]{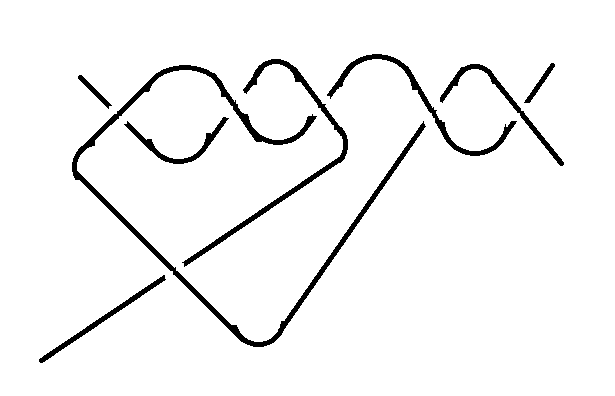} \qquad   \raisebox{1.9 em}{=}   \includegraphics[scale=0.25]{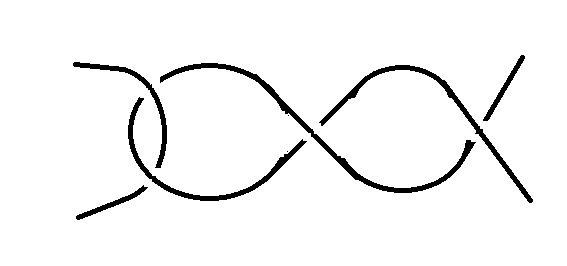}
\end{equation}
and
\begin{equation}\label{Ubar}
 \includegraphics[scale=0.3]{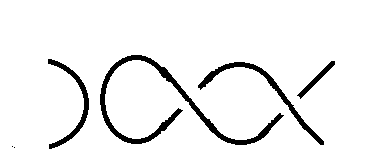}^{ \qquad \xrightarrow{\huge{\bar{\omega}}}} \qquad \includegraphics[scale=0.2]{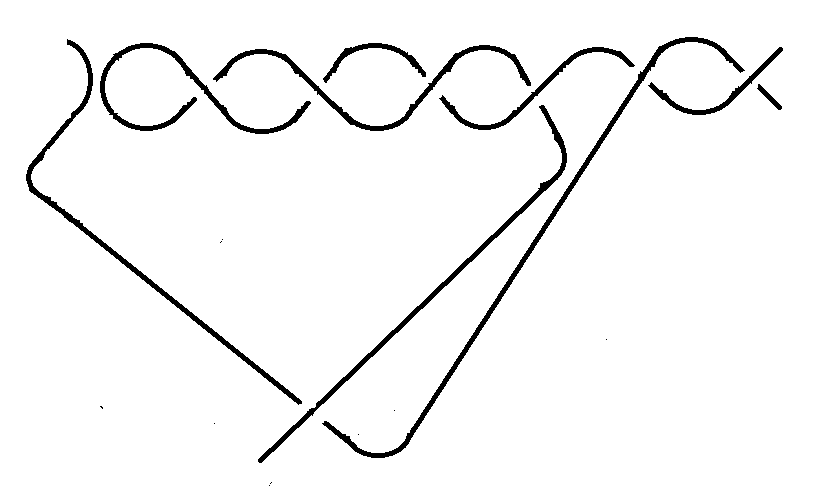} \qquad   \raisebox{1.9 em}{=}   \includegraphics[scale=0.3]{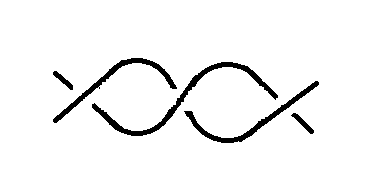}
\end{equation}

We therefore obtain the Thistlethwaite's link, figure  ~\ref{Thistlethwaite} by replacing the new tangles in figure   ~\ref{t-unlinked}. It is not difficult to see that the writhe of the Thistlethwaite's link is $-3$. This shows how the first example of Thistlethwaite fits the construction.

\subsection{A family of 2-component link with trivial Jones polynomial}

In this section we consider a sequence of 2-component links generated by $T = \infty - 2$ and $U = -T = \infty + 2$. The diagram $H(\infty - 2, \infty + 2)$ in figure ~\ref{generator}  is a 2-component link of writhe $0$; therefore a repeated application of the omega operation  on the tangle $T = \infty - 2$ yields a sequence of tangles $T_{0}=\infty - 2$ $T_{1}=3$, $T_{2}=5 \cdot 1 \cdot 2$, $T_{3}=5 \cdot 1 \cdot 4 \cdot 1 \cdot 2$,... such that $\langle H(T_{n},-T_{n})\rangle= \delta $ for all $n \geq 0$.

\begin{figure}[!h]
\centering 
\includegraphics[scale=0.4]{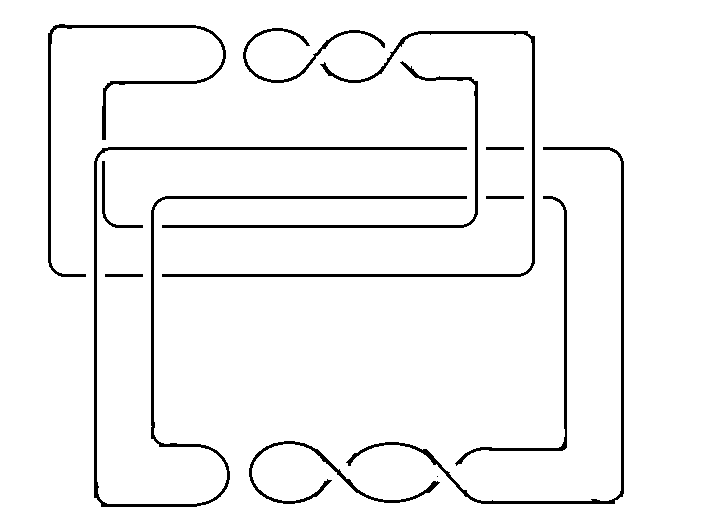}
\caption{$H(\infty - 2, \infty + 2)$}
\label{generator} 
\end{figure}

\begin{defn}[]
\label{defn:jwt}
$S(n)= H(T_{n}, -T_{n})$, where the tangle $T_{n}$ is the result of $n$ applications of the omega operation.
\end{defn}

It is easy to verify that the writhe of  $H(T_{n}, -T_{n})$ is zero for even $n$ and for odd $n$ the writhe is $\pm 8$, the sign depends on the choice of orientation.
\begin{thm}[]
\label{thm:jwt}
The Jones polynomial, $V_{S(n)}(t)= u$, for even $n$ and $V_{S(n)}(t)= t^{\pm 6}u$, for odd $n$ where the sign depending on the choice of orientation and $u= -t^{-\frac{1}{2}}-t^{\frac{1}{2}}$.
\end{thm}
\begin{proof}
For even $n$; writhe of $S(n)$ is zero.
\begin{eqnarray*}
V_{S(n)}(t) & = & (-t^{\frac{1}{4}})^{-3 \cdot 0}u \\
            & = & u
\end{eqnarray*}
and for odd $n$; writhe of $S(n)$ is $\pm 8$.
\begin{eqnarray*}
V_{S(n)}(t) & = & (-t^{\frac{1}{4}})^{-3 \cdot \pm 8}u \\
            & = & (-t)^{\pm 6}u\\
            & = &  t^{\pm 6}u
\end{eqnarray*}
\end{proof}
\begin{figure}[!ht]
    \subfloat[$S(1)$\label{subfig-1:dummy}]{%
      \includegraphics[scale=0.4]{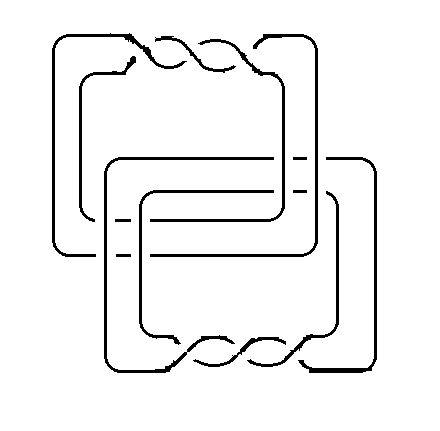}
    }
    \hfill
    \subfloat[$S(2)$\label{subfig-2:dummy}]{%
      \includegraphics[scale=0.4]{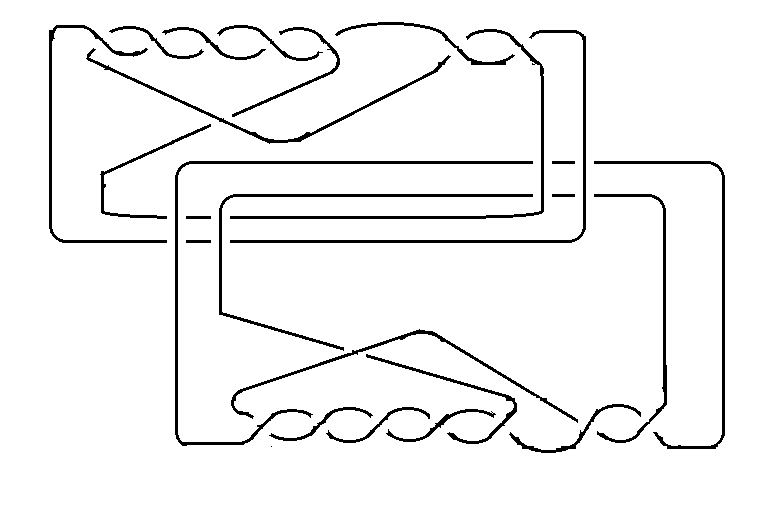}
    }
    \hfill
    \subfloat[$S(3)$\label{subfig-2:dummy}]{%
      \includegraphics[scale=0.4]{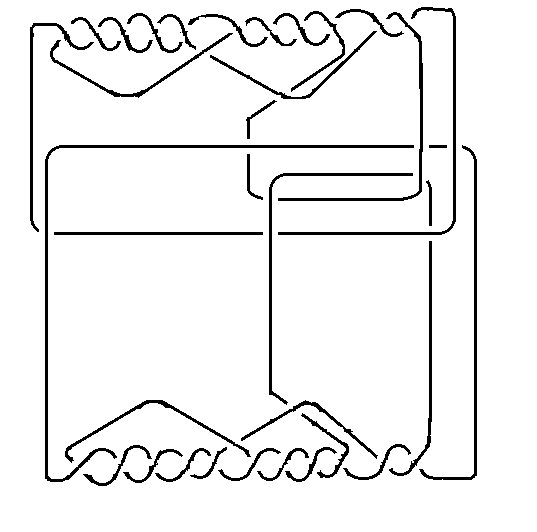}
    }
    \caption{The first three terms of the sequence, $S(n)$. }
    \label{Sequence}
  \end{figure}
  
 \newpage 
  The sequence $S(n)$ for positive even $n$ is therefore an infinite sequence of pairwise distinct non-trivial links whose linking cannot be detected by the Jones polynomial.
  
\bibliographystyle{plain}

\begin{thebibliography}{10}


\bibitem{b1} Collin C. Adams , \emph{The Knot Book},
W.H Freeman and Company; (1994).

\bibitem{a9}\label{some label - optional} D. Rolfsen, 
\emph{The quest for a knot with trivial {J}ones polynomial : diagram surgery and the {T}emperley-{L}ieb algebra,Topics in knot theory},
NATO Advanced Studies Institute Series C \textbf{399} (1993), 195--210.


\bibitem{a6}\label{some label - optional} J.W Alexander, 
\emph{Topological invariants of knots and links},
 Transactions of the American Mathematical Society \textbf{30} (1928), 275--306.
 
 \bibitem{a8}\label{some label - optional} K. Murasugi, 
\emph{The {J}ones polynomial and classical conjectures in knot theory},
Topology  \textbf{26} (1987), 187--194.


\bibitem{b3} Louis  H.  Kauffman, \emph{On {K}nots},
Princeton University Press, NJ; (1987).

\bibitem{b4}  Louis  H.  Kauffman , \emph{Formal {K}not {T}heory},
Princeton University Press, NJ; (1983).


\bibitem{b6} Louis H. Kauffman and Sostenes Lins , \emph{Temperley-{L}ieb {R}ecoupling {T}heory and {I}nvariants of 3-{M}anifolds},
Princeton University Press, NJ; (1994).


\bibitem{a12} Louis  H.  Kauffman, 
\emph{Notes on {D}ouble-{S}tranded {K}nots and {L}inks}, 1985,unpublished.
Available at \texttt{http://homepages.math.uic.edu/~kauffman/DblStrands.pdf}.

\bibitem{a1}\label{some label - optional} Louis  H.  Kauffman  , 
\emph{Knot {D}iagramatics-{H}and book of {K}not {T}heory},
 Elsevier Pub.,Amsterdam \textbf{} (2005), 233-318.


\bibitem{a3}\label{some label - optional} Louis  H.  Kauffman  , 
\emph{New Invariants in the theory of Knots},
 Amer. Math. Monthly \textbf{95} (1988), 85--97.


\bibitem{a7}\label{some label - optional} Louis H. Kauffman, 
\emph{State models of the {J}ones polynomial},
Topology  \textbf{26} (1987), 395--407.

\bibitem{a11}\label{some label - optional} Louis H. Kauffman, 
\emph{An invariant of regular isotopy},
Transactions of the American Mathematical Society \textbf{318} (1990), 417--471.

 
 \bibitem{b6} Louis H. Kauffman and Sostenes Lins , \emph{Temperley-{L}ieb {R}ecoupling {T}heory and {I}nvariants of 3-{M}anifolds},
Princeton University Press, NJ; (1994).


\bibitem{a3b}\label{some label - optional} Morwen Thistlethwaite, 
\emph{Links with trivial {J}ones Polynomial},
JKTR \textbf{10} No.1 (2001), 641--643.

\bibitem{b5}Richard H. Cromwell and Ralph H. Fox, \emph{Introduction to knot theory},
Springer-Verlag; (1963).

\bibitem{b2} Robert Messer and Philip Straffin , \emph{Topology Now},
The Mathematical Association of America; (2006).

\bibitem{a2}\label{some label - optional} S. Eliahou, Louis H. Kauffman and Morwen B. Thistlethwaite, 
\emph{Infinite Families of Links with trivial {J}ones Polynomial},
 Topology \textbf{42} (2003), 155-169.


\bibitem{a5}\label{some label - optional} V.F.R Jones, 
\emph{A polynomial invariant for knots via von Neumann algebras},
Bulletin of the American Mathematical Society \textbf{12} (1985), 103--111.

\bibitem{a10}\label{some label - optional} V.F.R Jones, 
\emph{A polynomial invariant for knot and von {N}eumann algebras},
 Notices American Mathematical Society \textbf{33} (1986), 218--225.
 


\end{thebibliography}

\end{document}